\pdfoutput=1
\RequirePackage{ifpdf}
\ifpdf 
\documentclass[pdftex]{sigma}
\else
\documentclass{sigma}
\fi

\usepackage{mathrsfs}
\usepackage{esint}

\DeclareMathAlphabet{\mathpzc}{OT1}{pzc}{m}{it}

\numberwithin{equation}{section}

\newtheorem{Theorem}{Theorem}[section]
\newtheorem{Corollary}[Theorem]{Corollary}
\newtheorem{Lemma}[Theorem]{Lemma}
\newtheorem{Proposition}[Theorem]{Proposition}
 { \theoremstyle{definition}
\newtheorem{Definition}[Theorem]{Definition}
\newtheorem{Example}[Theorem]{Example}
\newtheorem{Remark}[Theorem]{Remark} }

\newcommand{\rFs}[5]{\,_{#1}F_{#2}\!\left( \genfrac{.}{.}{0pt}{}{#3}{#4}
\,;#5 \right)}

\newcommand{\De}{\Delta}

\newcommand{\la}{\lambda}

\newcommand{\Om}{\Omega}

\def\smath#1{\text{\scalebox{.8}{$#1$}}}
\def\sfrac#1#2{\smath{\frac{#1}{#2}}}

\DeclareMathOperator\arctanh{arctanh}

\begin{document}

\allowdisplaybreaks

\newcommand{\arXivNumber}{1812.08553}

\renewcommand{\PaperNumber}{053}

\FirstPageHeading

\ShortArticleName{Orthogonal Dualities of Markov Processes and Unitary Symmetries}

\ArticleName{Orthogonal Dualities of Markov Processes\\ and Unitary Symmetries}

\Author{Gioia CARINCI~$^\dag$, Chiara FRANCESCHINI~$^\ddag$, Cristian GIARDIN\`A~$^\S$, \\
Wolter GROENEVELT~$^\dag$ and Frank REDIG~$^\dag$}

\AuthorNameForHeading{G.~Carinci, C.~Franceschini, C.~Giardin\'a, W.~Groenevelt and F.~Redig}

\Address{$^{\dag}$~Technische Universiteit Delft, DIAM, P.O.~Box 5031, 2600 GA Delft, The Netherlands}
\EmailD{\href{mailto:G.Carinci@tudelft.nl}{G.Carinci@tudelft.nl}, \href{mailto:W.G.M.Groenevelt@tudelft.nl}{W.G.M.Groenevelt@tudelft.nl}, \href{mailto:F.H.J.Redig@tudelft.nl}{F.H.J.Redig@tudelft.nl}}

\Address{$^\ddag$~Center for Mathematical Analysis Geometry and Dynamical Systems, Instituto \\
\hphantom{$^\ddag$}~Superior T\'ecnico, Universidade de Lisboa, Av. Rovisco Pais, 1049-001 Lisboa, Portugal}
\EmailD{\href{mailto:Chiara.Franceschini@tecnico.ulisboa.pt}{Chiara.Franceschini@tecnico.ulisboa.pt}}

\Address{$^\S$~University of Modena and Reggio Emilia, FIM, via G.~Campi 213/b, 41125 Modena, Italy}
\EmailD{\href{mailto:Cristian.Giardina@unimore.it}{Cristian.Giardina@unimore.it}}

\ArticleDates{Received December 24, 2018, in final form July 05, 2019; Published online July 12, 2019}

\Abstract{We study self-duality for interacting particle systems,
where the particles move as continuous time random walkers
having either exclusion interaction or inclusion interaction.
We show that orthogonal self-dualities arise from unitary symmetries
of the Markov generator. For these symmetries we provide two equivalent
expressions that are related by the Baker--Campbell--Hausdorff formula.
The first expression is the exponential of an anti Hermitian operator
and thus is unitary by inspection; the second expression is factorized
into three terms and is proved to be unitary by using generating functions.
The factorized form is also obtained by using an independent
approach based on scalar products, which is a new method
of independent interest that we introduce to derive (bi)orthogonal
duality functions from non-orthogonal duality functions.}

\Keywords{stochastic duality; interacting particle systems; Lie algebras; orthogonal polynomials}

\Classification{60J25; 82C22; 22E60}

\section{Introduction}

In a series of previous works, dualities that are {\em orthogonal} in an appropriate Hilbert space have been derived for a class of interacting particle systems with Lie-algebraic structure. This class includes several well-known processes, for instance the generalized exclusion processes \cite{Liggett, SS94}, the inclusion process \cite{GRV}, as well as independent random walkers~\cite{DP06}. These orthogonal dualities were identified as classical orthogonal polynomials in~\cite{FG17} by using the structural properties of those polynomials (recurrence relation and raising/lowering operators). In \cite{RS182} the approach of generating functions was used instead, by which non-polynomial orthogonal dualities (provided by some other special functions, e.g., Bessel functions) were also found. Orthogonal duality functions can also be explained using representation theory: they can be understood as the intertwiner between two unitarily equivalent representations of a Lie algebra~\cite{FGG, G17}.

Often the duality property of a Markov process can be related to the existence of some (hidden) symmetries of the Markov generator, i.e., operators commuting with the generator of the Markov process \cite{GKR, GKRV}. This occurs for instance when the process has a reversible measure. In this context detailed balance can be interpreted as a trivial duality, and by acting with a~symmetry of the generator one obtains a non-trivial duality. A natural question that arises is thus what type of symmetries lead to orthogonal dualities. In this paper, for some particular processes, we show that those symmetries have to be unitary and we single out the general expression they must have.

We expect the association between orthogonal dualities and unitary symmetries to be robust and apply in great generality to all cases where the duality function is obtained from the action of a symmetry on the trivial duality. We choose here to focus on a class of processes having an underlying Lie algebra structure that helps in the explicit characterization of the unitary symmetry. We thus consider three interacting particle systems (exclusion, inclusion, independent walkers) for which some orthogonal self-duality function are known and are given by classical discrete orthogonal polynomials as the Meixner, Krawtchouk and Charlier polynomials. For this class of processes we provide a full characterization of their unitary symmetries. This result allows to identify the entire family of orthogonal duality functions, which turns out to be a two-parameter family. For special values of the parameters we recover the orthogonal polynomial duality. We expect similar results could be found for higher orthogonal polynomials that would be associated to other Markov processes and their dualities.

The organization of this paper is as follows. In Section~\ref{sec2} we give an overview of the main tools required to construct the setting. In Section~\ref{2.1} we recall the concept of (self-)duality between Markov processes and we introduce the notion of equivalence between (self-)duality functions. In Section~\ref{ips} we introduce three algebras ($\mathfrak{su}(2)$ algebra, $\mathfrak{su}(1,1)$ algebra and the Heisenberg algebra) and the associated Markov processes that turn out to be interacting particle systems.
In Section~\ref{sec3} we recall from \cite{GKRV} a general scheme to construct duality functions for Markov processes whose generator has an algebraic structure. In this approach there is a one-to-one correspondence between self-duality functions and symmetries of the Markov generator. In Section~\ref{sec3.22}, by using this connection between duality functions and symmetries we present the first main result of this paper. Namely, in Theorem~\ref{lemort} we provide the expression for the most general unitary symmetry that will then yield orthogonal duality functions. We also identify the special values of the parameters appearing in these symmetries for which the duality functions are orthogonal polynomials. The proof of Theorem~\ref{lemort} is contained in Section~\ref{proof}. In Section~\ref{sec4} we provide a second expression for these unitary symmetries: it is a factorized expression for function of the algebra generators that we show to be connected to the previous expression via the Baker, Campbell, Hausdorff formula. In Section~\ref{sec5} we introduce a novel independent procedure to obtain orthogonal duality functions. This new method relies on the use of a scalar product in a Hilbert space. In Section~\ref{sec5.1} we prove that the scalar product of two duality functions is again a new duality function and in Section~\ref{bio} we show that these new duality functions are biorthogonal by construction. We apply this technique in Section~\ref{sipkof}: for the interacting particle systems considered in this paper by manipulation of the biorthogonal relation we get an orthogonal relation.

The literature on stochastic duality for Markov processes is extremely vast. For the reader convenience we recall \cite{BorCorSa14, CST16, ImSa11, KMP, Ohku16,Spohn} for some applications to non-equilibrium statistical physics, \cite{CGGR, M} for duality in population models and \cite{borodin2016stochastic, CGSL, corwin2016stochastic} for the study of singular stochastic PDE via duality. We also mention the algebraic approach to duality that shows that several Markov processes dualities in turn derive from algebraic structures, see for instance \cite{CGRS2,CGRS1, GKRV, kuan2018algebraic,Sch97}.

The orthogonal dualities that were alluded to at the beginning of this introduction have been introduced more recently in the literature. One might wonder what the advantages are of having orthogonality. Assume the process has state space $\mathscr{S}$ and invariant measure $\mu$ and consider the process generator as an operator on $\mathsf{L}^{2}(\mathscr{S}, \mu)$. The duality function can be viewed as a family of functions in the configurations of the original process (labelled by the configurations of the dual process). Then, when this family happens to be linearly independent and complete so that it gives a basis, it is natural to ask if/when this can be turned into a family of duality functions that are orthogonal, thus yielding an orthogonal basis. It is not clear a priori that the natural orthogonalization Gram--Schmidt procedure conserves the duality property. Thus this has to be checked independently. In all cases, having an orthogonal basis will be helpful in studying the contraction properties of the Markov semigroup, and thus quantifying for instance the rate of relaxation to the invariant measure. Furthermore, in~\cite{ACR} orthogonal duality has been used to prove a Boltzmann--Gibbs principle where several simplifications occur as a consequence of the fact that the duality functions constitute an orthogonal basis for the Hilbert space.

\section{Preliminaries} \label{sec2}
We start by recalling the definition of stochastic duality for two processes and introducing the algebras and the interacting particle system (IPS) of interest. Our goal is to describe a~constructive technique, in which self-duality functions arise from both the symmetric approach of Section~\ref{sec3} as well as from the inner product approach described in Section~\ref{sec5}.

\subsection{Stochastic duality}\label{2.1}
The definition of stochastic duality can be formalized for Markov processes as well as their infinitesimal generators. Although they are not equivalent in general, they become equivalent under suitable hypothesis regarding the semigroup associated to the generator of the process discussed in Proposition~1.2 of~\cite{JK}.

\begin{Definition}[Markov duality definitions]\label{dop}Let $X=(X_{t})_{t\geq 0}$ and $Y=(Y_{t})_{t\geq 0}$ be two continuous time Markov processes with state spaces $\mathscr{S}$ and $\mathscr{S}^{\rm dual}$ and generators $L$ and $L^{\rm dual}$ respectively. We say that $Y$ is {\em dual} to $X$ with duality function $D\colon \mathscr{S}\times \mathscr{S}^{\rm dual}\longmapsto \mathbb{R} $ if
\begin{gather*}
\mathbb{E}_{x}[D(X_{t},y)]=\mathbb{E}_{y}[D(x,Y_{t})] ,
\end{gather*}
for all $(x,y) \in \mathscr{S} \times \mathscr{S}^{\rm dual}$ and $t\geq 0$. If $X$ and $Y$ are two independent copies of the same process, we say that $Y$ is {\em self-dual} with self-duality function~$D$. Duality can also be regarded at the level of the processes generators. We say that $L^{\rm dual}$ is {\em dual} to $L$ with duality function $D\colon \mathscr{S}\times \mathscr{S}^{\rm dual}\longmapsto \mathbb{R} $ if
\begin{gather*}
[LD(\cdot,y)](x)=\big[L^{\rm dual}D(x,\cdot)\big](y) .
\end{gather*}
If $L=L^{\rm dual}$ we have self-duality.
\end{Definition}
Note that self-duality can always be thought as a special case of duality where the dual process is an independent copy of the first one. The simplification of self-duality for IPS typically arises from the fact that the computation of correlation functions of the original process reduces to studying a finite number of variables in the copy process.

\textbf{Countable state space.}
If the original process $(X_{t})_{t\geq 0}$ and the dual process $(Y_{t})_{t\geq 0}$ are Markov processes with countable state space $\mathscr{S}$ and $\mathscr{S}^{\rm dual}$ resp., then the duality relation is equivalent to
\begin{gather}\label{csspace}
\sum_{x' \in \mathscr{S}} L(x,x')D(x',y)=\sum_{y' \in \mathscr{S}} L^{\rm dual}(y,y')D(x,y')=\sum_{y' \in \mathscr{S}} \big(L^{\rm dual}\big)^{T}(y',y)D(x,y'),
\end{gather}
where $L^{T}$ denotes the transposition of the generator $L$. Generators are treated like (eventually infinite) matrices and
in matrix notation the identity~\eqref{csspace} becomes
\begin{gather}\label{dualitymatrix}
LD=D\big(L^{\rm dual}\big)^{T} .
\end{gather}
If $L^{\rm dual}=L$ we obtain the corresponding identities for self-duality. In this context, the genera\-tor~$L$ is given by a matrix known as \textit{rate matrix} such that
\begin{gather*}
L(x,y) \geq 0 \qquad \text{for} \quad x\neq y \qquad \text{and} \qquad \sum_{y} L(x,y) =0 .
\end{gather*}
We say that the process jumps from $x$ to $y$ with \textit{rate} $L(x,y)$.

\begin{Definition}[duality functions in product form and single site duality functions.]
The duality functions we will present turn out to be of the following product structure
\begin{gather*}
D({x},{y}) = \prod_{i } d(x_{i},y_{i})
\end{gather*}
for $(x,y) \in \mathscr{S}\times \mathscr{S}^{\rm dual} $. The function inside the product will be regarded as ``single site'' duality function and the subscript $i$ removed.
\end{Definition}
Throughout the paper we will work with duality functions of this structure and so we will only consider the single site.
\begin{Lemma}[notion of equivalence for duality functions.] \label{const} If $D(x,y)$ is a duality function between two processes and the function $c\colon \mathscr{S} \times \mathscr{S}^{\rm dual}\longrightarrow \mathbb{R}$ is constant under the dynamics of the two processes then $D_{c}(x,y)=c(x,y)D(x,y)$ is also a duality function. We will refer to $D$ and $D_{c}$ as equivalent duality functions.
\end{Lemma}
For example, in the context of the processes we are interested in, we will see that the dynamics conserves the total number of particles and dual particles, i.e., $\sum_{i} x_{i}=\sum_{i} y_{i}$ is conserved. As a consequence of this we can always choose a self-duality function up to a multiplicative factor in terms of the total number of particles. For example, if
\begin{gather*}
D(x,y) = \prod_{i=1}^{n} d(x_{i},y_{i})
\end{gather*}
is a self-duality function, then for constants $c$ and $b$, the function
\begin{gather*}
D_{b,c}(x,y) = \prod_{i=1}^{n} b^{x_{i}}c^{y_{i}}d(x_{i},y_{i})
\end{gather*}
is again a self-duality. This can easily be checked using Definition~\ref{dop}. Indeed,
\begin{align*}
\mathbb{E}_{x} \left(D_{b,c}(X(t),y) \right) & =
\mathbb{E}_{x} \left( \prod_{i=1}^{n} b^{X_{i}(t)}c^{y_{i}}d(X_{i}(t),y_{i}) \right) =
b^{\sum_{i}x_{i}} c^{\sum_{i} y_{i}} \mathbb{E}_{x} \left( D(X(t),y ) \right) \\
 & =
b^{\sum_{i}x_{i}} c^{\sum_{i} y_{i}} \mathbb{E}_{x} \left( D(x,Y(t) ) \right) =
\mathbb{E}_{y} \left( \prod_{i=1}^{n} b^{x_{i}}c^{Y_{i}(t)}d(x_{i},Y_{i}(t)) \right) \\
& =
\mathbb{E}_{y} \left(D_{b,c}(x,Y(t)) \right) .
\end{align*}
Our examples are all such that $b=1$ and so we omit it.

\subsection{Algebras and IPS}\label{ips}
In the next three sections we introduce three algebras with three IPS, each one corresponding to one of the three algebras. In particular, the probability measure that define the $*$-structure of the algebra turns out to be the reversible measure of the particle process associated to that algebra. Here we denote by $\mathscr{F}(\mathscr{S})$ the space of real-valued functions on~$\mathscr{S}$, with countable~$\mathscr{S}$.

\subsubsection[The Lie algebra $\mathfrak{su}(1,1)$ and symmetric inclusion process, ${\rm SIP}(k)$]{The Lie algebra $\boldsymbol{\mathfrak{su}(1,1)}$ and symmetric inclusion process, $\boldsymbol{{\rm SIP}(k)}$}\label{liealgebrasu11}

Generators of the \textit{dual} Lie algebra $\mathfrak{su}(1,1)$ are $K^{0}$, $K^{+}$ and $K^{-}$. They satisfy
\begin{gather*}
 \big[K^{0},K^{\pm}\big]=\mp K^{\pm} \qquad \text{and} \qquad \big[K^{+},K^{-}\big]=2K^{0} .
\end{gather*}
We shall work in a representation, labeled by $k\in\mathbb{R}_+$, where the actions of the three generators on functions $f$ in $\mathscr{F}(\mathbb{N}) $ is given by
\begin{gather}
( K^{+}f) (x):= (2k+x)f(x+1),\nonumber \\
( K^{-}f) (x):=xf(x-1), \label{discretesu11}\\
(K^{0}f) (x):=(x+k)f(x),\nonumber
\end{gather}
with $f(-1)=0$. We define an inner product on $\mathscr{F}(\mathbb N)$ by
\begin{gather} \label{normmeassip}
\langle f,g\rangle_{w_{p,k}}=\sum_{x} f(x)g(x) w_{p,k}(x) , \qquad w_{p,k}(x)= \frac{\Gamma(2k+x)}{x! \Gamma(2k)}p^{x}(1-p)^{2k} ,
\end{gather}
where $0<p<1$, then $\mathfrak{su}(1,1)$ acts on the corresponding Hilbert space $\mathsf{L}^{2}(w_{p,k})$ by unbounded operators with dense domain the set of finitely supported functions on $\mathbb N$. The adjoints of the generators with respect the inner product are given by
\begin{gather} \label{adjoints su11}
\big(K^{0}\big)^{ *}=K^{0}, \qquad (K^{+})^{ *}=\frac{1}{p}K^{-}, \qquad(K^{-})^{ *}=pK^{+} .
\end{gather}
The Casimir element is
\begin{gather*}
\Omega= 2 \big(K^{0}\big)^{2} -K^{+}K^{-}- K^{-}K^{+},
\end{gather*}
which is self-adjoint and commutes with every element of the Lie algebra.

The process associated with this algebra is the symmetric inclusion process ${\rm SIP}(2k)$, described below. The inclusion process is introduced first in \cite{GKR}, and also studied further in \cite{GKRV}. The ${\rm SIP}(2k)$ is a family of interacting particles processes labeled by the parameter $k> 0$ and that can be defined on a generic graph $G(V,E)$. The state space is unbounded so that each site can have an arbitrary number of particles. The ${\rm SIP}(2k)$ generator is
\begin{gather}\label{sipj}
L^{{\rm SIP}(2k)}= \sum_{\substack{1\le i < l \le |V| \\ (i,l)\in E}} L^{{\rm SIP}(2k)}_{i,l} , 		\\
 L^{{\rm SIP}(2k)}_{i,l} f({\bf x}) = x_{i}(2k+x_{l})\big[ f\big(\textbf{x}^{i,l}\big)-f(\textbf{x})\big]+x_{l}(2k+x_{i})\big[ f\big(\textbf{x}^{l,i}\big)-f(\textbf{x})\big] , \nonumber
\end{gather}
where ${\bf x}^{i,l}$ denotes the particle configuration obtained from the configuration ${\bf x}$ by moving one particle from site $i$ to site $l$, i.e., ${\bf x}^{i,l}={\bf x}-\delta_{i}+\delta_{l}$ and so the dynamic conserves the total number of particles. The generator can be defined on a weighted graph, however for the sake of simplicity we restrict here to~\eqref{sipj}, since the duality functions will not depend on the weights of the graph edges.

Clearly, the action of the generator involves only two connected sites and it can be produced with the representation~\eqref{discretesu11} acting on tensor products of $\mathscr{F}(\mathbb N)$ via the expression of the coproduct of the Casimir~$\Omega$. Recall that the coproduct is an algebra homomorphism denoted by~$\Delta$ and defined by
\begin{gather*}
\Delta(X)= X_1 + X_2 ,
\end{gather*}
for a Lie algebra element $X$. Here the subscript $i$ indicates that the operator acts in the $i$th factor of the tensor product. Higher-order coproducts are defined on Lie algebra elements $X$ by
\begin{gather*}
\Delta^n (X) = \sum_{i=1}^{n+1} X_i,
\end{gather*}
which we consider as an operator on $\mathscr{F}(\mathbb N)^{\otimes (n+1)}$. One can verify that for the couple of sites $(i,l)$ the generator of the ${\rm SIP}(2k)$ on two sites is written in terms of generators of the $\mathfrak{su}(1,1)$ Lie algebra as
\begin{gather*}
L^{{\rm SIP}(2k)}_{i,l} = K_{i}^{+}K_{l}^{-} + K_{i}^{-}K_{l}^{+}-2K_{i}^{0}K_{l}^{0}+2k^{2}= - \Delta (\Omega)_{i,l}+2 k^{2} .
\end{gather*}
This is an operator on $\mathscr{F}(\mathbb N)^{\otimes |V|}$, and the subscript $i,l$ indicates that $\Delta(\Omega)$ acts on the $i$th and $l$th factor of the tensor product. (Note that $\Delta(\Om)_{i,l} \neq \Om_i + \Om_l$). Since $\Omega$ commutes with every $X \in \mathfrak{su}(1,1)$ it follows that $L_{i,l}^{{\rm SIP}(2k)}$ commutes with $\De(X)_{i,l}=X_i+X_l$, and hence $L^{{\rm SIP}(2k)}= \sum L_{i,l}^{{\rm SIP}(2k)}$ commutes with $\De^{|V|-1}(X) = \sum X_i$.
Last, the reversible measure of the ${\rm SIP}(2k)$ process is given by the homogeneous product measure with marginals the Negative Binomial distributions with parameters $2k>0$ and $0 < p < 1$, i.e., with probability mass function $w_{p,k}$ of equation~\eqref{normmeassip}.

\subsubsection[The Lie algebra $\mathfrak{su}(2)$ and symmetric exclusion process, ${\rm SEP}(2j)$]{The Lie algebra $\boldsymbol{\mathfrak{su}(2)}$ and symmetric exclusion process, $\boldsymbol{{\rm SEP}(2j)}$}\label{liealgebrasu2}

Generators of the \textit{dual} $\mathfrak{su}(2)$ Lie algebra are $J^{0}$, $J^{+}$ and $J^{-}$ which satisfy the following commutation relations
\begin{gather*}
\big[J^{0},J^{\pm}\big]=\mp J^{\pm} \qquad \text{and} \qquad [J^{+},J^{-}]=-2J^{0} .
\end{gather*}
We work in a representation of $\mathfrak{su}(2)$ labeled by $j \in \mathbb N/2$ on functions $f$ in $\mathscr{F}(\left\lbrace 0,1, \ldots, 2j\right\rbrace )$ given by
\begin{gather}
( J^{+}f ) (x) := (2j-x)f(x+1),\nonumber \\
( J^{-}f ) (x) :=xf(x-1), \label{discretesu2}\\
\big(J^{0}f\big) (x) :=(x-j)f(x),\nonumber
\end{gather}
where $f(-1)=f(2j+1)=0$. Defining an inner product on $\mathscr{F}(\lbrace 0,1,\ldots,2j \rbrace)$ by
\begin{gather} \label{normmeassep}
\langle f,g \rangle_{w_{p,j}}=\sum_{x} f(x)g(x) w_{p,j}(x) , \qquad w_{p,j}(x)={2j \choose x} \left( \frac{p}{1-p} \right)^{x}(1-p)^{2j},
\end{gather}where $0<p<1$, the generators $J^0$, $J^+$ and $J^-$ act on the corresponding Hilbert space $\mathsf{L}^{2}(w_{p,j})$, with adjoints given by
\begin{gather*}
\big(J^{0}\big)^{ *}=J^{0},\qquad\ (J^{+})^{ *}=\frac{1-p}{p}J^{-}, \qquad (J^{-})^{ *}=\frac{p}{1-p}J^{+}.
\end{gather*}
The Casimir element is
\begin{gather*}
\Omega= 2 \big(J^{0}\big)^{2} +J^{+}J^{-}+ J^{-}J^{+},
\end{gather*}
which is self-adjoint and commutes with every element in the Lie algebra.

The process associated with this algebra is the exclusion process, defined below. For $j=1/2$ the boundary driven simple exclusion process has been studied using duality in~\cite{Spohn}. The model for arbitrary $j$ has been introduced and studied in~\cite{SS94}. The ${\rm SEP}(2j)$ is a family of interacting particles processes labeled by the parameter $j\in\mathbb{N}/{2}$ and that can be defined on the same graph $G(V,E)$, as before. Each site (vertex) of $G$ can have at most $2j$ particles and the ${\rm SEP}(2j)$ generator is
\begin{gather*}
L^{{\rm SEP}(2j)}= \sum_{\substack{1\le i < l \le | V | \\ (i,l)\in E}} L^{{\rm SEP}(2j)}_{i,l}, \\
 L^{{\rm SEP}(2j)}_{i,l} f({\bf x}) = x_{i}(2j-x_{l})\big[ f\big({\bf x}^{i,l}\big)-f({\bf x})\big]+(2j-x_{i})x_{l}\big[ f\big({\bf x}^{l,i}\big)-f({\bf x})\big] . \nonumber
\end{gather*}
As before we can write the generator of the ${\rm SEP}(2j)$ in two sites using the generators of the $\mathfrak{su}(2)$ algebra
\begin{gather*}
 L^{{\rm SEP}(2j)}_{i,l} = J_{i}^{+}J_{l}^{-} + J_{i}^{-}J_{l}^{+}+2J_{i}^{0}J_{l}^{0}-2 j^{2}= \Delta (\Omega)_{i,l}-2 j^{2} .
\end{gather*}
Last, the reversible measure of the ${\rm SEP}(2j)$ process is given by the homogeneous product measure with marginals the Binomial distribution with parameters $2j>0$ and $0<p<1$, i.e., with probability mass function $w_{p,j}$ of equation~\eqref{normmeassep}.

\subsubsection{The Heisenberg algebra and independent random walkers (IRW)}\label{heisenalgebra}
The dual Heisenberg algebra is the Lie algebra with generators $a$, $a^{\dagger}$ and $1$ such that
\begin{gather*}
 \big[a,a^{\dagger}\big]=-1 .
\end{gather*}
The Heisenberg algebra has a representation on $\mathscr{F}(\mathbb{N}) $ such that
\begin{gather}
 \big( a^{\dagger}f\big)(x)= f(x+1),\nonumber \\
 ( af )(x)=xf(x-1)\label{heisenbergdiscretegenerators}
\end{gather}
and $1$ acts as the identity, and where $f(-1)=0$.
Consider the inner product \begin{gather} \label{normmeasirw} \langle f,g\rangle_{w_{p}}=\sum_{x} f(x)g(x) w_{p}(x), \qquad w_{p}(x)= \frac{p^{x}}{x!}{\rm e}^{-p},
\end{gather}
where $p>0$, then the Heisenberg algebra acts on the corresponding Hilbert space $\mathsf L^2(w_p)$ by unbounded operators with dense domain the set of finitely support functions. The adjoints of $a$ and $a^\dagger$ with respect to the inner product are
\begin{gather*}
 a^{ *}=p a^{\dagger} \qquad \text{and} \qquad \big( a^{\dagger}\big)^{ *}=\frac{1}{p}a .
\end{gather*}
No such element as the Casimir is available for the Heisenberg algebra. The process associated with this algebra is the process of independent random walkers (IRW), which was first introduced in~\cite{spi} and is well-known. They are defined in the usual setting, the process consists of independent particles that perform a symmetric continuous time random walk at rate $1$ on the graph $G(V,E)$. The generator is given by
\begin{gather*}
L^{\rm IRW}= \sum_{\substack{1\le i < l \le |V| \\ (i,l)\in E}} L^{\rm IRW}_{i,l}, 	\\
L^{\rm IRW}_{i,l} f(\textbf{x}) = x_{i}\big[ f\big(\textbf{x}^{i,l}\big)-f(\textbf{x})\big]+x_{l}\big[ f\big(\textbf{x}^{l,i}\big)-f(\textbf{x})\big]. \nonumber
\end{gather*}
In terms of the generators of the Heisenberg algebra we have
\begin{gather*}
L^{\rm IRW}_{i,l} = a_i^\dagger a_l + a_i a_l^\dagger - a_i a_l^\dagger - a_i a_l^\dagger.
\end{gather*}
One can verify that $L^{\rm IRW}_{i,l}$ commutes with $\De(X)_{i,l}$ for every $X$ in the Heisenberg algebra, so that $L^{\rm IRW}$ commutes with $\De^{|V|-1}(X)$.
The reversible invariant measure is provided by a homogeneous product of Poisson distributions with parameter $p>0$, i.e., with probability mass function $w_{p}$ of equation~\eqref{normmeasirw}.

\subsection[Self-dualities via symmetries: general approach and classical self-dualities]{Self-dualities via symmetries: general approach\\ and classical self-dualities} \label{sec3}

A general scheme for constructing self-dualities of continuous time Markov processes whose ge\-ne\-rator has a symmetry, i.e., an operator commuting with its generator, has been first proposed in~\cite{GKRV}. In this section we first recall this approach and then we illustrate it by showing the symmetry that is associated to classical self-duality functions. By construction, we are guaranteed that the functions we find via symmetries are self-dualities, but not orthogonal. However, orthogonality can be inferred by proving that the symmetry is unitary. A family of unitary symmetries will be found in Section~\ref{sec3.2} and, by specializing to some values of the parameters, we will recover orthogonal dualities in terms of discrete orthogonal polynomials previously found in \cite{FG17, RS182}. This orthogonality task is also addressed in Section~\ref{sec5} where we show that biortho\-go\-nality can be achieved by construction.

Recall that, since our processes are defined on a countable state space $\mathscr{S}$, we can work with the notion of duality in matrix notation, namely equation~\eqref{dualitymatrix}.
\begin{Definition}
Let $A$ and $B$ be two matrices having the same dimension. We say that $A$ is a~symmetry of $B$ if $A $ commutes with $B$, i.e.,
\begin{gather*}
[A,B]=AB-BA=0 .
\end{gather*}
\end{Definition}
The main idea is that self-duality (in the context of Markov process with countable state space) can be recovered starting from a \textit{trivial duality} which is based on the reversible measures of the processes. Then the action of a symmetry of the model on this trivial self-duality give rise into a non-trivial one.
The following results, whose proof can be found in~\cite{GKRV}, formalize this idea.
\begin{Theorem}[symmetries and self-duality]\label{symm+sd} Let $d$ be a self-duality function of the generator $L$ and let $S$ be a symmetry of $L$, then $D=Sd$ is again a self-duality function for $L$.
\end{Theorem}
If there is a description of the process generator in terms of a Lie algebra, then symmetries can be constructed using this algebraic structure. The two main elements of Theorem~\ref{symm+sd} are the initial self-duality $d$ and the symmetry operator~$S$. In general, if the process has a reversible measure the self-duality $d$ can easily be found starting from the reversibility.

\begin{Lemma}[diagonal self-duality and reversibility]\label{cheaptrivialduality} If the process associated to generator $L$ has reversible measure $\mu$, then the diagonal self-duality functions are of the form
\begin{gather*}
d(x,y)=\frac{\delta_{x,y}}{\mu(x)},\qquad \text{where} \quad x,y \in \mathscr{S} .
\end{gather*}
\end{Lemma}

We refer to these diagonal self-duality functions as trivial or ``cheap'' self-duality functions. The next lemma summarizes the cheap self-dualities for our three processes: notice that, up to neglectable factors, they are the inverse of their reversible measure.

\begin{Lemma}[trivial self-duality functions]\label{chlemma}The processes of interests are self-dual with single site diagonal self-duality function given by
\begin{gather*}
D^{\rm ch}_{p}(x,y)=
\begin{cases}
\dfrac{y! \Gamma(2k)}{\Gamma(2k+y)} p^{-y} \delta_{x,y} &\text{for the} \ {\rm SIP}(2k),	\vspace{1mm}\\
 \dfrac{(2j-y)!y!}{2j!}\left( \frac{1-p}{p}\right) ^{y}	\delta_{x,y} &\text{for the} \ {\rm SEP}(2j),\vspace{1mm}\\
 \dfrac{y!}{p^{y}} \delta_{x,y} &\text{for the} \ {\rm IRW}.
 \end{cases}
\end{gather*}
\end{Lemma}

We can now find several self-duality results applying the recipe of Theorem \ref{symm+sd} using the trivial self-duality function as starting point. We will give symmetries in terms of the exponential function of Lie algebra elements considered as operators on function spaces, see~\eqref{discretesu11}, \eqref{discretesu2} and~\eqref{heisenbergdiscretegenerators}. Recall that our representations are defined in terms raising/lowering operators (i.e., shift by $\pm 1$) and diagonal operators. These operators are defined on finitely supported functions, so that $n$th powers of the raising and lowering operators will always be $0$ for large $n$. Consequently, when considering exponential functions of raising and lowering operators acting on finitely supported functions we do not have to worry about convergence of series, but in other cases we have to check convergence. The exponential function of a diagonal operator is again a diagonal operator. Moreover, it will be enough to provide symmetries acting on functions on one site. Indeed, we have shown that the generator $L$ of the process, which acts on $N$ sites, commutes with $\De^{N-1}(X^n)$ for any $n\in \mathbb N$ and any Lie algebra element $X$. As a consequence $L$ also commutes with $\exp\big(\De^{N-1}(X)\big) = \exp(X_1)\exp(X_2) \cdots \exp(X_N)$.

The following lemma shows how to find the so-called classical self-duality functions which have a lower triangular structure.
\begin{Proposition}[classical self-duality functions and associated symmetries]\label{clalemm} The following results hold.
\begin{enumerate}\itemsep=0pt
\item[$1.$] The ${\rm SIP}(2k)$ is self-dual with single site self-duality function given by
\begin{gather} \label{sipcla}
D^{\rm cl}_{p}(x,y):=S\big( D^{\rm ch}_{p} ( \cdot, y ) \big) (x) = \frac{x!}{(x-y)!} \frac{\Gamma(2k)}{\Gamma(2k+y)} p^{-y} \mathbf{1}_{\lbrace y \leq x \rbrace },
\end{gather}
where $S={\rm e}^{K^{-}}$.

\item[$2.$] The ${\rm SEP}(2j)$ is self-dual with single site self-duality function given by
\begin{gather*}
D^{\rm cl}_{p}(x,y):=S\big( D^{\rm ch}_{p} (\cdot, y) \big) (x) = \frac{x!}{(x-y)!} \frac{ (2j-y)!}{2j!} \left( \frac{1-p}{p} \right) ^{y} \mathbf{1}_{\lbrace y \leq x \rbrace },
\end{gather*}
where $ S={\rm e}^{J^{-}}$.

\item[$3.$] The ${\rm IRW}$ is self-dual with single site self-duality function given by
\begin{gather*}
D^{\rm cl}_{p}(x,y):=S\big( D^{\rm ch}_{p} (\cdot, y) \big) (x) = \frac{x!}{(x-y)!} \frac{1}{p^{y}}\mathbf{1}_{\lbrace y \leq x \rbrace } ,
\end{gather*}
where $ S={\rm e}^{a}$.
\end{enumerate}
\end{Proposition}
\begin{proof}We only consider the first item, the proof for the other two is similar. The fact that $D^{\rm cl}_{p}(x,y)$ is a self-duality function is an immediate consequence of Theorem \ref{symm+sd} since ${\rm e}^{K^-}$ commutes with the Casimir~$\Omega$. The second equality in~\eqref{sipcla} follows from a straightforward calculation. Indeed, acting with the symmetry $S$, we have
\begin{align*}
D^{\rm cl}(x,y) &={\rm e}^{K^{-}} \big( D^{\rm ch}_{p} (\cdot, y) \big) (x) = \sum_{i=0}^{\infty} \frac{(K^{-})^{i}}{i!}\frac{y!\Gamma(2k)}{\Gamma(2k+y)} \left( \frac{1}{p} \right) ^{y} \delta_{x,y} \\
& = \sum_{i=0}^{\infty} \frac{y!}{i!}\frac{\Gamma(2k)}{\Gamma(2k+y)} \frac{x!}{(x-i)!} \left( \frac{1}{p} \right) ^{y} \mathbf{1}_{\lbrace i \leq x \rbrace }\delta_{x-i,y} \\ & = \frac{x!}{(x-y)!} \frac{\Gamma(2k)}{\Gamma(2k+y)} \left( \frac{1}{p} \right) ^{y} \mathbf{1}_{\lbrace y \leq x \rbrace } .\tag*{\qed}
\end{align*}\renewcommand{\qed}{}
\end{proof}

By virtue of Lemma \ref{const} one can either neglect constants and factors that are constant under the dynamic of the process or, on the other hand, add convenient choice of these constant factors. In particular, in Section~\ref{sec5}, we will fix the value of these constants in a suitable way.

\section{Orthogonal self-dualities and unitary symmetries}\label{sec3.2}
In what follows, we will relate the orthogonal polynomials with their hypergeometric functions. In general, the hypergeometric functions $\mathstrut_r F_s$ is defined as an infinite series
\begin{gather*} \rFs{r}{s}{a_{1},\ldots,a_{r}}{b_{1},\ldots,b_{s}}{x}= \sum_{k=0}^{\infty}\frac{(a_{1})_{k}\cdots(a_{r})_{k}}{(b_{1})_{k}\cdots(b_{s})_{k}}\frac{x^{k}}{k!},
\end{gather*}
where $( a )_{k}$ denotes the Pochhammer symbol defined in terms of the Gamma function as
\begin{gather*}
( a)_{k}:=\frac{\Gamma(a+k)}{\Gamma(a)} .
\end{gather*}
Whenever one of the numerator parameters is a negative integer, the hypergeometric function~$\mathstrut_r F_s $ turns into a finite sum, so it is a polynomial in the other numerator parameters. We define polynomials as in~\cite{KLS}, in particular the following three discrete polynomials: Meixner polynomials
\begin{gather*}
M(x,y;p)= \rFs{2}{1}{-x, -y}{2k }{1-\frac{1}{p}} \qquad \text{for} \quad x,y \in \mathbb{N},
\end{gather*}
Krawtchouk polynomials
\begin{gather*}
K(x,y;p)= \rFs{2}{1}{-x, -y}{-2j }{\frac{1}{p}} \qquad \text{for} \quad x,y = 0,1,\ldots, 2j,
\end{gather*}
and the Charlier polynomials
\begin{gather*}
C(x,y; p)= \rFs{2}{0}{-x, -y}{-}{-\frac{1}{p}} \qquad \text{for} \quad x,y \in \mathbb{N}.
\end{gather*}

\subsection{Main result}\label{sec3.22}
In this section we explicitly determine the symmetries $S$, given in terms of the underlying Lie algebra generators, which allow to retrieve the orthogonal polynomials. It is important to mention that, since we start from a (trivial) self-duality which is orthogonal with respect to the measure~$w$, the operator $S$ that produces the orthogonal self-duality must be unitary. Recall that a unitary operator in $\mathsf{L}^{2}(\mathscr{S}, w)$ is a linear operator such that
\begin{gather*}
UU^{\ast}=U^{\ast}U=I ,
\end{gather*}
where $U^{\ast}$ is the adjoint of $U$ in $\mathsf{L}^{2}(\mathscr{S}, w)$. As a consequence of this, we will have that $U$ preserves the inner product of the Hilbert space $\mathsf{L}^{2}(\mathscr{S}, w)$ and so the norm of the cheap self-duality function~$D^{\rm ch}$ must be the same of the norm of the orthogonal self-duality function $D^{\rm or}=SD^{\rm ch}$
\begin{gather*}
\big\| D^{\rm or} \big\|_{w}^{2} = \langle S D^{\rm ch},S D^{\rm ch} \rangle_{w}=\big\| D^{\rm ch} \big\|_{w}^{2} .
\end{gather*}
In the spirit of Proposition~\ref{clalemm} we list the new orthogonal symmetries for the interacting particles systems.
\begin{Theorem}[orthogonal self-duality functions and associated symmetries]\label{lemort} The following results hold.
\begin{enumerate}\itemsep=0pt
\item[$1.$] For the ${\rm SIP}(2k)$ we have that
\begin{enumerate}\itemsep=0pt
\item[$i)$] The symmetry
\begin{gather}\label{tiledSsip}
S_{\alpha, \beta}= \exp \left( \beta \left(-K^{+} + \frac{1}{p} K^{-} \right) \right) \exp\big({\rm i} \alpha K^{0} \big)
\end{gather}extends to a unitary operator for every choice of $ \alpha, \beta \in \mathbb{R}$.
As a consequence the functions $S_{\alpha, \beta} \big( D^{\rm ch}_{p} ( x, \cdot) \big) (y) $ are orthogonal $($single site$)$ self-duality functions in $\mathsf{L}^{2}(w_{p,k})$ with squared norm $\big\| D^{\rm ch}_{p} \big\|^{2}_{w_{p,k}}$.
\item[$ii)$] Choosing $\alpha=\hat{\alpha} = \pi$ and $\beta=\hat{\beta} = \sqrt{p} \arctanh \big( \sqrt{p} \big) $ we get the Meixner polynomials up to a constant: $D^{\rm or}_{p}(x,y):= S_{\hat{\alpha}, \hat{\beta} } \big( D^{\rm ch}_{p} (x,\cdot) \big) (y) =(p-1)^{k} M(x,y;p) $.
\end{enumerate}

\item[$2.$] For the ${\rm SEP}(2j)$ we have that
\begin{enumerate}\itemsep=0pt
\item[$i)$] The symmetry
\begin{gather}\label{tiledSsep}
S_{\alpha, \beta}= \exp \left( \beta \left(-J^{+} + \frac{1-p}{p} J^{-} \right) \right) \exp\big({\rm i} \alpha J^{0} \big)
\end{gather}is unitary for every choice of $ \alpha, \beta \in \mathbb{R}$.
As a consequence the functions \linebreak $S_{\alpha, \beta} \big( D^{\rm ch}_{p}(x,\cdot ) \big) (y) $ are orthogonal $($single site$)$ self-duality functions in $\mathsf{L}^{2}(w_{p,j})$ with squared norm $\big\| D^{\rm ch}_{p}\big\|^{2}_{w_{p,j}}$.
\item[$ii)$] Choosing $\alpha=\hat{\alpha} = \pi$ and $\beta=\hat{\beta} = \sqrt{\frac{p}{1-p}} \arctan \Big( \sqrt{\frac{p}{1-p}} \Big) $ we get the Krawtchouk polynomials up to a constant: $D^{\rm or}_{p}(x,y):= S_{\hat{\alpha}, \hat{\beta} } \big( D^{\rm ch}_{p} ( x, \cdot) \big) (y) = (p-1)^{j} K(x,y;p) $.
\end{enumerate}

\item[$3.$] For the ${\rm IRW}$ we have that
\begin{enumerate}\itemsep=0pt
\item[$i)$] The symmetry
\begin{gather}\label{tiledSirw}
S_{\alpha, \beta}= \exp \big( \beta \big( -p a^{\dagger} + a \big) \big) \exp\big({\rm i} \alpha a a^{\dagger} \big)
\end{gather}extends to a unitary operator for every choice of $\alpha, \beta \in \mathbb{R}$.
As a consequence the functions $S_{\alpha, \beta} \big( D^{\rm ch}_{p} ( x, \cdot) \big) (y) $ are orthogonal $($single site$)$ self-duality functions in $\mathsf{L}^{2}(w_{p})$ with squared norm $\big\| D^{\rm ch}_{p} \big\|^{2}_{w_{p}}$.
\item[$ii)$] Choosing $\alpha=\hat{\alpha} = \pi$ and $\beta=\hat{\beta} = 1 $ we get the Charlier polynomials up to a constant: $D^{\rm or}_{p}(x,y):= S_{\hat{\alpha}, \hat{\beta} } \big( D^{\rm ch}_{p} ( x, \cdot) \big) (y) = {\rm e}^{- \frac{p}{2}} C(x,y;p) $.
\end{enumerate}
\end{enumerate}
\end{Theorem}

\subsection{Proof of the main result}\label{proof}
We need the following lemma to introduce the generating function and to compute the action of the algebra generators in order to prove Theorem \ref{lemort}. In particular we only consider the $\mathfrak{su}(1,1)$ algebra and the ${\rm SIP}(2k)$ process, but for the other two processes the idea is the same.
\begin{Definition}[generating functions] We will always use the definition of generating function as in \cite[formula (9.10.11)]{KLS}, i.e., the generating function $G$ of $g(y)$ is defined as
\begin{gather*}
 ( Gg ) (t) := \sum_{y=0}^{\infty} g(y) \frac{\Gamma(2k+y)}{y!\Gamma(2k)} t^{y} , \qquad t \in \mathbb{R} .
\end{gather*}\end{Definition}
The generating function of Meixner polynomials $M(x,y;p)$ (see~\cite{KLS}) is
\begin{gather}\label{kragf}
 \sum_{y=0}^{\infty} M(x,y;p) \frac{\Gamma(2k+y) }{y! \Gamma(2k) } t^{y} = \left( 1- \frac{t}{p} \right)^{x} (1-t)^{-2k-x} .
 \end{gather}
\begin{Lemma}[intertwining of the $\mathfrak{su}(1,1)$ algebra generators]\label{toperator}
The following results hold
\begin{enumerate}\itemsep=0pt
\item[$1.$] $GK^{-}g(y) = \big(2k t + t^{2} \frac{\partial}{\partial t} \big) Gg(t) =: \mathpzc{K}^{-}Gg(t) $.
\item[$2.$] $GK^{+}g(y)= \big( \frac{\partial}{\partial t} \big) Gg(t) =: \mathpzc{K}^{+}Gg(t)$.
\item[$3.$] $GK^{0}g(y)= \big( k + t \frac{\partial}{\partial t} \big) Gg(t) =: \mathpzc{K}^{0}Gg(t)$.
\end{enumerate}
Note that $ \mathpzc{K}^{-} $, $ \mathpzc{K}^{+} $ and $ \mathpzc{K}^{0}$ so defined satisfy the commutation relations of the dual $\mathfrak{su}(1,1)$ Lie algebra.
\end{Lemma}
\begin{proof}
We have
\begin{align*}
GK^{-}g(y) &= \sum_{n=0}^{\infty} yg(y-1) \frac{\Gamma(2k+y)}{y!\Gamma(2k)} t^{y} \\ & = 2k t \sum_{y=0}^{\infty} g(y) \frac{\Gamma(2k+y)}{y!\Gamma(2k)} t^{y} + t^{2} \sum_{y=0}^{\infty} g(y) \frac{\Gamma(2k+y)}{y!\Gamma(2k)} y t^{y-1} \\ & =
\left( 2k t + t^{2} \frac{\partial}{\partial t} \right) Gg(t) = \mathpzc{K}^{-}Gg(t) .
\end{align*}
This implicitly defines the operator $\mathpzc{K}^{-}$ which acts on functions of the $t$ variable as
\begin{gather*}\mathpzc{K}^{-}:= 2k t+t^{2}\frac{\partial}{\partial t} .
\end{gather*}
Similarly,
\begin{align*}
GK^{+}g(y) & = \sum_{y=0}^{\infty} (2k+y)g(y+1) \frac{\Gamma(2k+y)}{y!\Gamma(2k)} t^{y} \\
&= \sum_{y=0}^{\infty} g(y) \frac{\Gamma(2k+y)}{y!\Gamma(2k)} yt^{y-1}
 =\left( \frac{\partial}{\partial t} \right) Gg(t) = \mathpzc{K}^{+}Gg(t) ,
\end{align*}
so the operator $\mathpzc{K}^{+}$ is a first derivative with respect to $t$, defined as
\begin{gather*}
\mathpzc{K}^{+}f(t):= \frac{\partial f}{\partial t} (t) .
\end{gather*}
For $K^{0}$ we proceed in the same way
\begin{gather*}
GK^{0}g(y) = \sum_{y=0}^{\infty} (k+y) g(y)\frac{\Gamma(2k+y)}{y!\Gamma(2k)} t^{y} =
\left( k+t \frac{\partial}{\partial t} \right) Gg(t) = \mathpzc{K}^{0}Gg(t) ,
\end{gather*}
and we infer that
\begin{gather*}
\mathpzc{K}^{0}f(t):= \left( k + t \frac{\partial}{\partial t} \right) f(t) .
\end{gather*}
Note that for all the above we have called $f(t)= ( Gg(\cdot) ) (t)$.
\end{proof}

\begin{proof}[Proof of Theorem \ref{lemort}] We will only give a proof for the first item as the other two follow a~similar strategy. The first point of the first item regards the unitarity of $S_{\alpha, \beta}$ in $\mathsf{L}^{2}(w_{p,k})$, which is achieved if $ ( S_{\alpha, \beta} )^{\ast} = ( S_{\alpha, \beta} )^{-1} $. Using the adjoints~\eqref{adjoints su11} of $K^0$, $K^+$ and $K^-$ we have that $ ( S_{\alpha, \beta} )^{\ast} = \exp \big( {-}{\rm i} \alpha K^{0} \big) \exp \big( \beta \big({-}\frac{1}{p} K^{-} + K^{+} \big) \big) = ( S_{\alpha, \beta} )^{-1}$ as an operator acting on the space of finitely supported functions. Since this is dense in $\mathsf{L}^2(w_{p,k})$, $S_{\alpha, \beta}$ extends to a unitarity operator.
Unitary operators conserve the norm and so the norm of $S_{\alpha, \beta} D^{\rm ch}_{p} (x,y)$ is the same as the norm of
 $D^{\rm ch}_{p}(x,y)$ in $\mathsf{L}^{2}(w_{p,k})$. In particular, the two squared norms are
 \begin{gather*}
 \big\| S_{\alpha, \beta} D^{\rm ch}_{p} \big\|^{2}_{w_{p,k}} =\big\| D^{\rm ch}_{p} \big\|^{2}_{w_{p,k}}= \frac{y! \Gamma(2k)}{\Gamma(2k+y) } p^{-y} (1-p)^{2k} .
 \end{gather*}
We show now the proof of the second point using a generating function approach. The idea is to show that the generating function of $D^{\rm or}_{p}=S_{\hat{\alpha}, \hat{\beta}} D^{\rm ch}_{p}$ and the Meixner polynomials are the same, i.e.,
\begin{gather} \label{genkra}
G\big( S_{\hat{\alpha}, \hat{\beta} } \big( D^{\rm ch}_{p} ( x, \cdot) \big) \big) (y)= G \bigl ( (p-1)^{k} M(x,\cdot ;p) \bigr) (y)
 \end{gather} and so using the generating function of Meixner polynomials in equation~\eqref{kragf} one has that the r.h.s.\ of equation~\eqref{genkra} is $(p-1)^{-k} (1-t)^{-2k-x} \big(1- \frac{t}{p} \big)^{x} $.
For the l.h.s.\ instead of computing $ G \big( S_{\alpha, \beta} D^{\rm ch}_{p} \big) (t)$ we use Lemma~\ref{toperator} to evaluate $ \mathpzc{S}_{\alpha, \beta} \big( G D^{\rm ch}_{p} \big) (t) $, here
\begin{gather*}
 \mathpzc{S}_{\alpha, \beta} = \exp \left( \beta \left(-\mathpzc{K}^{+} + \frac{1}{p} \mathpzc{K}^{-}\right) \right) \exp \big( {\rm i} \alpha \mathpzc{K}^{0} \big),
\end{gather*}
where $\mathpzc{K}^{+}$, $\mathpzc{K}^{-}$ and $\mathpzc{K}^{0} $ are those in Lemma \ref{toperator}, which we consider as operators on functions that are analytic at $0$.
In other words, we have to find the action of the operator $ \mathpzc{S}_{\alpha, \beta}$ on
\begin{gather*}
\big( G D^{\rm ch}_{p} \big) (t) = \sum_{y=0}^{\infty} \frac{y! \Gamma(2k)}{\Gamma(2k+y)} p^{-y} \delta_{x,y} \frac{\Gamma(2k+y)}{y! \Gamma(2k)} t^{y} = \left( \frac{t}{p} \right) ^{x} .
\end{gather*}
The action of $\exp \big( {\rm i} \alpha \mathpzc{K}^{0} \big)$ on $f(t)=Gg (t)$ is
\begin{gather*}
 \big( {\rm e}^{{\rm i} \alpha} \big)^{\mathpzc{K}^{0}} f(t) := G\Big( \big( {\rm e}^{{\rm i} \alpha} \big)^{K^{0}} g \Big) (t) = \sum_{y=0}^{\infty} \frac{\Gamma(2k+y)}{y! \Gamma(2k)} t^{y} ({\rm e}^{{\rm i} \alpha} )^{y+k} g(y) = ({\rm e}^{{\rm i} \alpha} )^{k} f\big( {\rm e}^{{\rm i} \alpha}t \big) .
\end{gather*}
Letting $\alpha= \hat{\alpha}= \pi$ one has
\begin{gather}\label{neweq}
\exp \big( {\rm i} \pi \mathpzc{K}^{0} \big) f(t) = (-1)^{k} f(-t) .
\end{gather}To find the action of $ \exp \big( \beta \big({-}\mathpzc{K}^{+} + \frac{1}{p} \mathpzc{K}^{-}\big) \big) $ we will solve a partial differential equation, whose solution $\psi(t,\beta) $ is
the action of $ \mathpzc{S}_{\alpha, \beta}$ on function $f(t)$. Using Lemma~\ref{toperator}, this is
\begin{gather}\label{fi}
\psi(t, \beta) = {\rm e}^{\beta \big[ \big( \frac{t^{2}}{p} -1 \big) \frac{\partial }{\partial t} + \frac{2k}{p} t \big] } f(t)
\end{gather}
with initial condition $\psi(t, 0) = f(t)$. Here, it is understood that for the operator $B:= \big( \frac{t^{2}}{p} -1 \big) \frac{\partial }{\partial t} + \frac{2k}{p} t$ and a function $f$ in its domain, the exponential ${\rm e}^{\beta B} f$ is defined as the solution of the partial differential equation $\frac{\partial}{\partial \beta} g(\beta,t)= B g(\beta,t)$ with initial condition $g(0,t)=f(t)$. Thus, deriving both sides of~\eqref{fi} with respect to $\beta$, we get a first-order PDE for~$\psi$:
\begin{gather}\label{fii}
\frac{\partial \psi }{\partial \beta} - \left( \frac{t^{2}}{p} -1 \right) \frac{\partial \psi }{\partial t} - \frac{2kt}{p} \psi =0.
\end{gather}To solve the PDE we use the method of characteristics: we consider $ \psi $ along the characteristic plane $(\tau, s)$, so
that along a characteristic curve $\tau$ is constant and $ \psi(t,\beta) = \psi(t(s),\beta(s)) $. We then have
\begin{gather*}
\frac{\partial \psi}{\partial s} = \frac{\partial \psi }{\partial \beta} \frac{\partial \beta }{\partial s} + \frac{\partial \psi }{\partial t} \frac{\partial t}{\partial s} .
\end{gather*}
Comparing the above with the PDE in equation \eqref{fii} we just have to solve a system of three first-order ODEs:
\begin{gather*}
\frac{\partial \beta }{\partial s} = 1, \\
\frac{\partial t}{\partial s} = \frac{p-t^{2}}{p},\\
\frac{\partial \psi }{\partial s} = \frac{2k t}{p} \psi.
\end{gather*}
From the first equation we have immediately that $\beta = s $, while the second has solution
\begin{gather*}
t(s) = \sqrt{p} \frac{\tanh\big(s /\sqrt{p} \big) + \tanh (c_{1})}{1+ \tanh\big(s /\sqrt{p} \big) \tanh (c_{1})} .
\end{gather*}
Using the initial condition $t(0)= \sqrt{p} \tanh(c_{1}) = \tau$ we get $c_{1}=\arctanh\big( \tau/ \sqrt{p}\big) $ and so
\begin{gather*}
t(s) = \sqrt{p} \frac{ \tau/ \sqrt{p} + \tanh\big(s /\sqrt{p} \big)}{1+\tau/ \sqrt{p} \tanh\big(s /\sqrt{p} \big)} .
\end{gather*}
Substituting $t$ in the last ODE we find that
\begin{gather*}
\psi(s)=\big( \tau \sinh\big(s/ \sqrt{p}\big)+ \sqrt{p} \cosh\big(s/ \sqrt{p}\big) \big)^{2k} c_{2} .
\end{gather*}
To find $c_{2}$ we use the initial condition in the characteristic plane, i.e.,
$\psi(0) = f(\tau)=p^{k}c_{2}$ so $c_{2}=\frac{f(\tau)}{p^{k}}$ and so our solution in the $(\tau, s)$ plane is
\begin{gather*}
\psi(\tau, s)=f(\tau) \left( \frac{\tau}{\sqrt{p}} \sinh\big(s/ \sqrt{p}\big)+ \cosh\big(s/ \sqrt{p}\big) \right)^{2k} .
\end{gather*}
In the $(t, \beta)$ plane this becomes
\begin{gather*}
\psi(t, \beta)= f\left( \sqrt{p} \frac{t- \sqrt{p} \tanh \big(\beta / \sqrt{p}\big)}{\sqrt{p} - t \tanh \big(\beta / \sqrt{p}\big)} \right) \left( - \frac{t}{\sqrt{p}} \sinh\big(\beta/ \sqrt{p}\big) + \cosh\big(\beta/ \sqrt{p}\big)\right) ^{-2k} .
\end{gather*}
Setting $\beta=\hat{\beta} = \arctanh \big( \sqrt{p} \big) \sqrt{p} $ the above expression simplifies to
\begin{gather}\label{thisonee}
{\rm e}^{\hat{\beta} \big[ \big( {-}1 + \frac{t^{2}}{p}\big) \frac{\partial }{\partial t} + \frac{2k}{p} t\big] } f(t) = \left(\frac{1-t}{\sqrt{1-p}} \right) ^{-2k} f\left( \frac{t-p}{1-t}\right) .
\end{gather}
Equation \eqref{neweq} together with \eqref{thisonee} finally gives
\begin{gather}\label{thisoneee}
 \mathpzc{S}_{\hat{\alpha}, \hat{\beta}}f(t) = (p-1)^{k}(1-t)^{-2k} f\left( \frac{p-t}{1-t} \right) .
\end{gather}
Last, we need to set $f(t)=\big( G D^{\rm ch}_{p}\big) (t)= \big( \frac{t}{p}\big) ^{x}$ to finally get
\begin{gather*}
 \mathpzc{S}_{\hat{\alpha}, \hat{\beta}}\big( G D^{\rm ch}_{p} \big) (t) = (p-1)^{k} (1-t)^{-2k-x} \left( 1- \frac{t}{p}\right)^{x} ,
\end{gather*}
which matches the generating function of the Meixner polynomials.
\end{proof}

In the following section we give a different expression for the three unitary symmetries $S_{\hat{\alpha},\hat{\beta}}$ of Theorem~\ref{lemort}.

\subsection{Factorized symmetries}\label{sec4}
We now want to study the unitary symmetries that arise from the previous section. Since we do not know how to act with these symmetries on functions $f(x) \in\mathscr{F}(\mathbb{N})$, we wonder if a~`factorized' version of $S_{\hat{\alpha},\hat{\beta}}$ exists, i.e., if we can find~$a$,~$b$ and $c$ such that
\begin{gather*}
S_{\hat{\alpha},\hat{\beta}}={\rm e}^{aK^{-}}{\rm e}^{bK^{0}}{\rm e}^{cK^{+}} .
\end{gather*}
The advantage of having a factorized symmetry is that one can directly compute its action on~$f(x)$ (without passing via generating functions), even if, on the other hand, the unitary property is not an immediate consequence of this form. In the next section we will relate this factorized form to another symmetry.

\begin{Theorem}[factorized unitary symmetries] \label{factsymm} The three orthogonal symmetries $S_{\hat{\alpha},\hat{\beta}}$ can also be written in a factorized version using the appropriate algebra generators.
\begin{enumerate}\itemsep=0pt
\item[$1.$] The action of $S_{\hat{\alpha},\hat{\beta}}$ in equation~\eqref{tiledSsip} coincides with the action of ${\rm e}^{K^{-}}{\rm e}^{\log(p-1)K^{0}}{\rm e}^{p K^{+}}$.
\item[$2.$] The action of $S_{\hat{\alpha},\hat{\beta}}$ in equation~\eqref{tiledSsep} coincides with the action of ${\rm e}^{J^{-}}{\rm e}^{\log\big( \frac{1}{p-1}\big) J^{0}}{\rm e}^{\frac{p}{1-p}J^{+}}$.
\item[$3.$] The action of $S_{\hat{\beta}}$ in equation~\eqref{tiledSirw} coincides with the action of ${\rm e}^{a}{\rm e}^{-p/2+{\rm i} \pi a a^{\dagger}}{\rm e}^{p a^{\dagger}}$.
\end{enumerate}
\end{Theorem}
\begin{proof}We only show the first item as the other two have similar proofs; to do that we still use generating functions. To show that
\begin{gather}\label{equalitysymmetry}
{\rm e}^{K^{-}}(p-1)^{K^{0}}{\rm e}^{p K^{+}} g(y) = \exp \left( \hat{\beta} \left(-K^{+} + \frac{1}{p} K^{-} \right) \right)\exp \big( {\rm i} \hat{\alpha}K^{0}\big) g(y) ,
\end{gather}we first consider the generating function $G$ on both sides and then flip the action of $G$ with the one of the operators to get
\begin{gather}\label{thisone}
{\rm e}^{\mathpzc{K}^{-}} (p-1)^{\mathpzc{K}^{0}} {\rm e}^{p \mathpzc{K}^{+}} f(t) = \exp \left( \hat{\beta} \left(-\mathpzc{K}^{+} + \frac{1}{p} \mathpzc{K}^{-} \right) \right) \exp \big( {\rm i} \hat{\alpha} \mathpzc{K}^{0} \big) f(t) ,
\end{gather}
where we called $f(t) = ( G g ) (t) $ and $ \mathpzc{K}^{-} $, $ \mathpzc{K}^{0} $ and $ \mathpzc{K}^{+} $ are those in Lemma~\ref{toperator}.
The r.h.s.\ of equation~\eqref{thisone} has been evaluated in the proof of Theorem~\ref{lemort}, equation \eqref{thisoneee} so we just need to find the action of $ {\rm e}^{\mathpzc{K}^{-}} $, $ (p-1)^{\mathpzc{K}^{0}} $ and $ {\rm e}^{p \mathpzc{K}^{+}} $ .
Clearly,
\begin{gather*}
 (p-1)^{\mathpzc{K}^{0}} f(t) = (p-1)^{k} f (t (p-1) )
\end{gather*}
since
\begin{gather*}
 (p-1)^{\mathpzc{K}^{0}} f(t) := G\Big( (p-1)^{K^{0}} g(y)\Big)\\
 \hphantom{(p-1)^{\mathpzc{K}^{0}} f(t)}{} = \sum_{y=0}^{\infty} \frac{\Gamma(2k+y)}{y! \Gamma(2k)} t^{y} (p-1)^{y+k} g(y) = (p-1)^{k} f ( t(p-1) ) .
\end{gather*}
For $ {\rm e}^{\mathpzc{K}^{-}} $ one can solve the associated PDE as in the proof of Theorem \ref{lemort}, or equivalently considering the limit as $p\rightarrow 0$ on both sides of equation~\eqref{thisonee} and using that $ \lim\limits_{p\rightarrow 0} \frac{\arctanh(\sqrt{p})}{\sqrt{p}}=1$
leads to
\begin{gather*}
{\rm e}^{\mathpzc{K}^{-}} f(t) = (1-t)^{-2k} f\left( \frac{t}{1-t} \right) .
\end{gather*}
Last, for $ {\rm e}^{p \mathpzc{K}^{+}} $ we have that
\begin{gather*}
\big( {\rm e}^{p \mathpzc{K}^{+}} f \big) (t) = {\rm e}^{ p \frac{\partial}{\partial t} } f (t) = f ( t + p )
\end{gather*}
since the action of the first derivative is a shift. Acting on $f(t)$, we have
\begin{gather*}
{\rm e}^{\mathpzc{K}^{-}}(p-1)^{\mathpzc{K}^{0}} {\rm e}^{p \mathpzc{K}^{+}} f(t) = {\rm e}^{\mathpzc{K}^{-}}(p-1)^{\mathpzc{K}^{0}} f ( t + p ) =
(p-1)^{k}{\rm e}^{\mathpzc{K}^{-}} f ( t(p-1) + p ) \\
\hphantom{{\rm e}^{\mathpzc{K}^{-}}(p-1)^{\mathpzc{K}^{0}} {\rm e}^{p \mathpzc{K}^{+}} f(t)}{} =(p-1)^{k}(1-t)^{-2k} f \left( \frac{t}{1-t}(p-1) + p \right)\\
\hphantom{{\rm e}^{\mathpzc{K}^{-}}(p-1)^{\mathpzc{K}^{0}} {\rm e}^{p \mathpzc{K}^{+}} f(t)}{} = (p-1)^{k}(1-t)^{-2k} f \left( \frac{p-t}{1-t} \right) ,
\end{gather*}
which matches the action of $S_{\hat{\alpha},\hat{\beta}}$ in equation~\eqref{thisoneee}.
\end{proof}

\begin{Remark}[Baker--Campbell--Hausdorff formula for dual $\mathfrak{su}(1,1)$ algebra] The identity given in equation \eqref{equalitysymmetry} can also be established as a consequence of the Baker--Campbell--Hausdorff formula for the $\mathfrak{su}(1,1)$ algebra, see \cite[formula~(24b)]{truax1985} adapted to the dual $\mathfrak{su}(1,1)$ algebra. In formula~(24b) one has the following replacement $L_{+}=\frac{1}{\sqrt{p}}K^{-}$, $L_{-}=\sqrt{p}K^{+}$ and $L_{0}=K^{0}$ and in particular one has to set $\tau=\arctanh(\sqrt{p})$ and $\alpha= \pi$.
\end{Remark}
The added value of having the factorized version of the symmetry $S_{\hat{\alpha},\hat{\beta}} $ is that one can immediately verify its action on the cheap duality
$D^{\rm ch}_{p}(x,y)$: via a straightforward computation one can produce the orthogonal polynomials of Theorem \ref{lemort}, as we show in the proposition below.
\begin{Proposition}[direct computation of orthogonal polynomials]\label{prop11}
Acting with the factorized symmetry on the cheap self-duality function one gets the orthogonal self-duality function. In particular, for the ${\rm SIP}(2k)$ this is
\begin{gather*}
{\rm e}^{K^{-}}{\rm e}^{\log(p-1)K^{0}}{\rm e}^{p K^{+}} \big( D^{\rm ch}_{p} ( x, \cdot) \big) (y) = D^{\rm or}_{p}(x,y).
\end{gather*}
\end{Proposition}
\begin{proof}
The proof follows a straightforward computation, see the appendix.
\end{proof}

\section{Orthogonal self-duality via scalar products} \label{sec5}
In this section we first show how duality and self-duality function emerge as a consequence of what we call scalar product approach and which is introduced below. We then give some hypo\-the\-sis to guarantee that such self-duality functions are biorthogonal. To conclude we implement this recently developed technique to find Meixner polynomials as orthogonal self-duality functions for the ${\rm SIP}(2k)$, in a similar way one could find orthogonal self-dualities for ${\rm SEP}(2j)$ and~IRW.
\subsection{Scalar product approach} \label{sec5.1}
In this section we present a new technique to approach duality: the naive idea is that the scalar product of two duality functions is still a duality function. We define the scalar product on some measure space $\mathsf{L}^{2}(\mathscr{S}, \mu)$, in the usual way, i.e.,
\begin{gather*}
\langle f,g \rangle_{\mu} = \sum_{x \in \mathscr{S}} f(x) g(x) \mu(x) .
\end{gather*}
We will show that -- in the setting of reversible processes -- once two
duality relations are available then it is possible to generate new different duality functions starting from the initial ones.
Suppose we have three processes with generators $L_{1}$, $L_{2}$ and $L_{3}$ and state space $ \mathscr{S}_{1} $, $ \mathscr{S}_{2} $ and~$\mathscr{S}_{3} $, respectively.
In particular, assume that $d_{1}$ is a duality function for $L_{1}$ and $L_{2}$, while $d_{2}$ is a duality function for $L_{3}$ and $L_{2}$, i.e.,
\begin{gather}\label{11}
L_{1}d_{1}(\cdot , y) (x) = L_{2}d_{1}(x, \cdot) (y) \qquad \text{for} \quad (x,y) \in \mathscr{S}_{1} \times \mathscr{S}_{2}
\end{gather}and
\begin{gather}\label{12}
L_{3}d_{2}(\cdot , y) (x) = L_{2}d_{2}(x, \cdot) (y) \qquad \text{for} \quad (x,y) \in \mathscr{S}_{3} \times \mathscr{S}_{2} .
\end{gather}
Then the following proposition holds.

\begin{Proposition}[new duality functions] \label{sca-du} If $\mu$ is a reversible measure for the generator $L_{2}$ and if equations \eqref{11} and \eqref{12} hold, then the function $D\colon \mathscr{S}_{1} \times \mathscr{S}_{3} \rightarrow \mathbb{R} $, given by
\begin{gather*}D(x,y) = \langle d_{1}(x,\cdot),d_{2}(y,\cdot) \rangle_{\mu}
\end{gather*}is a duality function for $L_{1}$ and $L_{3}$. If $L_{1}=L_{2}=L_{3}=L$, then $D$ is a new self-duality function for $L$.
\end{Proposition}
\begin{proof} For $i=1,2,3$, $L_{i,x}D(x,y)$ stands for $(L_{i}D(\cdot,y))(x)$ the action of $L_i$ on the $x$ variable of $D$. Then,
\begin{align*}
L_{1, x}D(x,y) &= \langle L_{1,x}d_{1}(x,\cdot),d_{2}(y,\cdot) \rangle_{\mu} = \sum_{z} L_{2,z} d_{1}(x,z) d_{2}(y,z) \mu(z) \\ &= \sum_{z} d_{1}(x,z) L_{2, z}d_{2}(y,z) \mu(z) = \langle d_{1}(x,\cdot), L_{3, y}d_{2}(y, \cdot) \rangle_{\mu} = L_{3, y}D(x,y) ,
\end{align*}
where we use duality of $d_1$ (resp.~$d_2$) in the second (resp.\ fourth) equality and the self-adjointness of $L_2$ with respect to~$\mu$.
\end{proof}

A first application of the above proposition is shown in the example below, where we recover Laguerre polynomials as duality function between ${\rm SIP}(2k)$ and ${\rm BEP}(2k)$, which we do not introduce here, but it is well explained in \cite[Section~2.2]{CGGR2}.
\begin{Example}[duality via scalar product]\label{ex1} A parametrized family of reversible measure for the ${\rm SIP}(2k)$ process is
\begin{gather*}
\mu_{p}(z)=\frac{\Gamma(2k+z)}{ \Gamma(2k) z! }p^{z}, \qquad z \in \mathbb{N}, \qquad p \in (0,1)
\end{gather*}
and the classical self-duality function $D^{\rm cl}_{p}$ for ${\rm SIP}(2k)$ is in equation \eqref{sipcla}, it will be our $d_{1}(x,y)$. The last ingredient we need is a duality function between ${\rm BEP}(2k)$ and ${\rm SIP}(2k)$, a well known in the literature, see \cite[equation~(4.9)]{CGGR2}, is
\begin{gather*}
d_{2}(x,y)=\frac{x^{y}\Gamma(2k)}{\Gamma(2k+y)}(-1)^{y} , \qquad x \in \mathbb{R^{+}}, \qquad y \in \mathbb{N}.
\end{gather*}
In particular $d_{2}$ is the one we need to obtain Laguerre polynomials. Proposition \ref{sca-du} assures us that $D(x,y)= \langle d_{2}(x, \cdot), d_{2}(y, \cdot) \rangle_{\mu_p} $ is a duality function between ${\rm SIP}(2k)$ and ${\rm BEP}(2k)$ and a straightforward computation shows that $D$ is the closed form of the Laguerre polynomials. Indeed,
\begin{align*}
D(x,y)=\langle d_{2}(x, \cdot), d_{1}(y, \cdot) \rangle_{\mu_p} & = \sum_{z=0}^{\infty} \frac{(-x)^{z} \Gamma(2k)}{\Gamma(2k+z)} \frac{y!}{(y-z)!} \frac{\Gamma(2k)}{ \Gamma(2k+z)} p^{-z} \frac{\Gamma(2k+z)}{\Gamma(2k) z!} p^{z} \\
& = \sum_{z=0}^{y} \frac{(-x)^{z}}{z!} \frac{y!}{(y-z)!} \frac{\Gamma(2k)}{\Gamma(2k+z)} =\rFs{1}{1}{-y}{2k }{ x }
\end{align*}
for $y \in \mathbb{N}$ and $x \in \mathbb{R}^{+}$.
\end{Example}

We can apply Proposition \ref{sca-du} for the same generator, to construct the Meixner polynomials as ${\rm SIP}(2k)$ self-duality functions.
\begin{Example} [self-duality via scalar product] As for the previous Example \ref{ex1}, let $\mu_{p}(z)$ be the reversible measure for the ${\rm SIP}(2k)$ process. Consider now two classical self-duality functions~$d_1$ and~$d_2$ as in equation~\eqref{sipcla}. In particular, we are free to choose them without the constant, i.e.,
\begin{gather*}
d_1(x,y)=d_2(x,y)=\frac{x!}{(x-y)!} \frac{\Gamma(2k)}{\Gamma(2k+y)} p ^{-y} \mathbf 1_{y\leq x} .
\end{gather*}
 A simple computation shows that their scalar product in $\mathsf{L}^{2}(\mu_{p})$ is a Meixner polynomial. Indeed,
\begin{align*}
D(x,y)&= \langle d_2(x, \cdot), d_1(y, \cdot) \rangle_{\mu_{p}} \\
& = \sum_{z=0}^{\infty} \frac{x!}{(x-z)!} \frac{y!}{(y-z)!} \left( \frac{\Gamma(2k)}{\Gamma(2k+z)} \right) ^{2} p^{-2z} \mathbf 1_{z\leq x} \mathbf 1_{z\leq y}\cdot \frac{\Gamma(2k+z)}{\Gamma(2k) z!} p^{z} \\
& = \sum_{z=0}^{x\wedge y} \frac{1}{z!} \frac{x!}{(x-z)!} \frac{y!}{(y-z)!} \frac{\Gamma(2k)}{\Gamma(2k+z)} p^{-z}
= M ( x,y; 1-p ) \qquad \text{for} \quad x, y \in \mathbb{N}.
\end{align*}
\end{Example}
The following proposition expands the result of Proposition \ref{sca-du} in the context of self-duality. It turns out that when two self-duality functions, $d$ and $D$, are in a relation via a scalar product with a third function $F$, then, assuming $d$ to be a basis for $\mathsf{L}^{2}(\mathscr{S}, \mu)$, $F$ must also be a self-duality function.
\begin{Proposition}[basis and self-duality]\label{propscalar1}
Assume that $\left\lbrace x \mapsto d(x,n) \,|\, n \in \mathscr{S} \right\rbrace $ is a basis of self-duality functions for $\mathsf{L}^{2}(\mathscr{S}, \mu)$ where $\mu$ is a reversible measure for the generator $L$. Let $F=F(n,z)$ be a function on $\mathscr{S} \times \mathscr{S}$ and define $D$ by
\begin{gather*}
D(x,n) := \langle d(x, \cdot), F(n, \cdot) \rangle_{\mu} .
\end{gather*}
If $D$ is self-duality function, so is $F$.
\end{Proposition}
\begin{proof}Using the short notation we have that
\begin{gather*}
L_{x}D(x,n)= \langle L_{x}d(x, \cdot), F(n, \cdot) \rangle_{\mu} = \sum_{z \in \mathscr{S}}d(x,z) L_{z} F(n,z) \mu(z),
\end{gather*}
where we used that $d$ is self-duality and that $L$ is self-adjoint with respect to~$\mu$. On the other hand, since $D$ is a self-duality the above quantity must be equal to
\begin{gather*}
L_{n}D(x,n)= \langle d(x, \cdot), L_{n}F(n, \cdot) \rangle_{\mu} = \sum_{z \in \mathscr{S}}d(x,z) L_{n} F(n,x) \mu(z) .
\end{gather*}
From the identity $L_{x}D(x,n)=L_{n}D(x,n)$, we have
\begin{gather*}
\sum_{z \in \mathscr{S}} d(x,z) \left[ L_{z} F(n,z)- L_{n} F(n,z) \right] \mu(z)=0
\end{gather*}
and since $d$ is a basis for $\mathsf{L}^{2}(\mathscr{S}, \mu)$, necessarily $ L_{z} F(n,z)- L_{n} F(n,z)=0 $, i.e., $F$ is also a~self-duality function for~$L$.
\end{proof}

\subsection{Biorthogonal self-dualities}\label{bio}
How does the orthogonality property play a role? Not all self-duality functions built with this method turn out to be orthogonal. However, there is a sort of stability with respect to this orthogonal property in the scalar product construction. More precisely, if we start with two biorthogonal self-duality functions the scalar product construction yields novel biorthogonal self-duality functions that may happen to be \textit{equal} and therefore orthogonal.

To state the next proposition, we will use that the inverse of the reversible measure is a~self-duality function as shown in Lemma~\ref{cheaptrivialduality}.
\begin{Proposition}[biorthogonal self-duality functions] \label{prop:orthogonality} Let $\mu_1$ and $\mu_2$ be two reversible measures and $d_1$, $d_2$ be two self-duality functions for the Markov process with generator $L$. Suppose that
\begin{gather}\label{d12}
\langle d_1(x,\cdot), {d}_2(\cdot,n) \rangle_{\mu_1}=\frac{\delta_{x,n}}{\mu_2(n)} \qquad\text{and}\qquad\langle d_2(x,\cdot), {d}_1(\cdot,n) \rangle_{\mu_2}=\frac{\delta_{x,n}}{\mu_1(n)} .
\end{gather}
Then the functions
\begin{gather*}
{ D}(x,n) := \langle d_1(x,\cdot), d_1(n,\cdot)\rangle_{ \mu_1}, \qquad \widetilde D(x,n) := \langle d_2(\cdot,x) , d_2(\cdot,n) \rangle_{\mu_1}
\end{gather*}
 are biorthogonal self-duality functions for $L$, i.e.,
\begin{gather*}
\big\langle D(\cdot,m), \widetilde{D} (\cdot,n) \big\rangle_{\mu_2} = \frac{ \delta_{m,n}}{\mu_2(m)} .
\end{gather*}
In particular, if $D = \widetilde{ D}$ we have the orthogonality relations for $D$.
\end{Proposition}
\begin{proof}
From Proposition \ref{sca-du} we have that both $D$ and $\widetilde{D}$ are self-duality functions since scalar product of self-dualities.
Assuming now we can interchange the order of summation:
\begin{align*}
\big\langle D(\cdot,m), \widetilde{ D}(\cdot,n)\big\rangle_{\mu_2} &= \sum_{x} D(x,m) \widetilde{ D}(x,n) \mu_2(x) \\
& = \sum_{x} \left(\sum_y d_1(x,y)d_1(m,y) \mu_1(y)\right) \left( \sum_z {d}_2(z,x) d_2(z,n) \mu_1(z)\right)\mu_2(x)\\
& = \sum_{y,z} { d}_1(m,y) d_2(z,n) \mu_1(y)\mu_1(z)\sum_x { d}_2(z,x) d_1(x,y)\mu_2(x) \\
& = \sum_{y,z} \mu_1(y)\mu_1(z) { d}_1(m,y) d_2(z,n) \frac{ \delta_{y,z}}{\mu_1(y)} \\
& = \sum_y { d}_1(m,y) d_2(y,n) \mu_1(y) = \frac{ \delta_{m,n}}{\mu_2(m)}.\tag*{\qed}
\end{align*}\renewcommand{\qed}{}
\end{proof}

We now implement this method to get the result below. Here we find Meixner polynomials as biorthogonal self-duality functions and with the aid of some hypergeometric functions transformation we find an orthogonal duality function.

\subsection{From biorthogonal to orthogonality self-duality functions}\label{sipkof}
According to Proposition \ref{prop:orthogonality} we need two duality functions $d_1$ and $d_2$ satisfying \eqref{d12}. For ${\rm SIP}(2k)$ recall that
\begin{gather*}
\mu_p(z):=\frac{\Gamma(2k+z)}{\Gamma(2k) z!} p^{z}
\end{gather*}
is the marginal of the (product) reversible (non normalized) measure,
and
\begin{gather*}
D^{\rm cl}_{\sfrac{1}{\la}}(x,n):= \frac{x! \Gamma(2k) \lambda^{y}}{(x-y)! \Gamma(2k+y)} \mathbf 1_{ y \leq x}
\end{gather*}
are the (single-site) classical self-duality functions. In the following we denote by $\langle\cdot,\cdot\rangle_p$ the scalar product with respect to the non-normalized measure $\mu_p$. We have the following lemma which we show for ${\rm SIP}(2k)$.
\begin{Lemma}[input self-duality functions] \label{lem:biorthogonality}
For any $p,q\in \mathbb R$ we have
\begin{gather*}
\big\langle D^{\rm cl}_{-p}(x,\cdot), D^{\rm cl}_{-q}(\cdot,n)\big\rangle_{p} = \frac{\delta_{x,n}}{\mu_q(x)},
\end{gather*}
where $D^{\rm cl}$ are the classical self-duality functions introduced in Proposition~{\rm \ref{clalemm}}.
\end{Lemma}
\begin{proof}Note that $D^{\rm cl}_{\sfrac{1}{\la}}(x,y) = 0$ if $y > x$. Then
\begin{gather*}
\big\langle D^{\rm cl}_{\sfrac{1}{\la}}(x,\cdot), D^{\rm cl}_{\sfrac{1}{\alpha}}(\cdot,n)\big\rangle_{p} = \sum_{z=n}^x \frac{x! \Gamma(2k) \lambda^{z}}{(x-z)! \Gamma(2k+z)} \frac{z! \Gamma(2k) \alpha^{n}}{(z-n)! \Gamma(2k+n)} \frac{p^{y}\Gamma(2k+z)}{\Gamma(2k) z!},
\end{gather*}
which equals 0 if $x < n$. Suppose $x \geq n$, then shifting the summation index ($m=z-n$) gives
\begin{gather*}
\big\langle D^{\rm cl}_{\sfrac{1}{\la}}(x,\cdot), D^{\rm cl}_{\sfrac{1}{\alpha}}(\cdot,n)\big\rangle_p = \frac{(p\la\alpha)^n \Gamma(2k)}{\Gamma(2k+n)} \sum_{m=0}^{x-n} \frac{ x!}{(x-n-m)!m!}(\la p)^m .
\end{gather*}
Now use $(A)_{k+l} = (A)_k (A+k)_l$, i.e.,
$ \frac{\Gamma(A+k+l)}{\Gamma(A)}=\frac{\Gamma(A+k)}{\Gamma(A)} \frac{\Gamma(A+k+l)}{\Gamma(A+k)}$ and the binomial theorem to obtain
\begin{align*}
\big\langle D^{\rm cl}_{\sfrac{1}{\la}}(x,\cdot), D^{\rm cl}_{\sfrac{1}{\alpha}}(\cdot,n)\big\rangle_p &= \frac{x!}{(x-n)!} \frac{\Gamma(2k)}{\Gamma(2k+n)} (p \la \alpha)^{n} \sum_{m=0}^{x-n} \frac{(x-n)!}{(x-n-m)!m!}(\la p)^m \\
 &=\frac{x!}{(x-n)!} \frac{\Gamma(2k)}{\Gamma(2k+n)} (p \la \alpha)^{n} (1+\la p)^{x-n} .
\end{align*}
Setting $\la= -\frac{1}{p} $ and $\alpha=-\frac 1 q$ we get the result, i.e.,
\begin{gather*}
\big\langle D^{\rm cl}_{-p}(x,\cdot), D^{\rm cl}_{-q}(\cdot,n)\big\rangle_p =\frac{x! \Gamma(2k)}{\Gamma(2k+x)} \left(+\frac{1}{q} \right)^{x} \delta_{x,n} = \frac{\delta_{x,n}}{\mu_{q}(x)} .\tag*{\qed}
\end{gather*}\renewcommand{\qed}{}
\end{proof}

\begin{Proposition}[from biorthogonal to orthogonality self-duality functions] The self-duality functions
\begin{gather*}
\begin{split}
D(x,n)=\big\langle D^{\rm cl}_{-q}(x,\cdot), D^{\rm cl}_{-q}(n,\cdot) \big\rangle_{q}
\end{split}
\end{gather*}
and
\begin{gather*}
 \widetilde D(x,n)=\big\langle D^{\rm cl}_{-p}(\cdot,x), D^{\rm cl}_{-p}(\cdot,n) \big\rangle_{q}
\end{gather*}
are biorthogonal with respect to the stationary measure of their associated process. In details,
\begin{enumerate}\itemsep=0pt
\item[$1.$] For ${\rm SIP}(2k)$ we have
\begin{gather*}
D(x,n)=\rFs{2}{1}{-x,-n}{2k}{\frac{1}{q}}
\end{gather*}
and
\begin{gather*}
\widetilde D(x,n)=\left( \frac{q}{p(1-q)}\right)^{n+x} (1-q)^{-2k} \rFs{2}{1}{-x,-n}{2k}{\frac{1}{q}}
\end{gather*}
and they are biorthogonal with respect to the measure $w_{p,k}$. In particular, for the choice
\begin{gather*}
\frac 1 q=1-\frac 1 p
\end{gather*}
we have
\begin{gather*}
\begin{split}
D(x,n)=M(x,n;p)\qquad \text{and} \qquad \widetilde D(x,n)=(1-p)^{2k} D(x,n)
\end{split}
\end{gather*}
so that the biorthogonality relation recovers the orthogonality relation of Meixner polynomials.

\item[$2.$] For ${\rm SEP}(2j)$ we have that
\begin{gather*}
D(x,n)=\rFs{2}{1}{-x,-n}{-2j}{\frac{q-1}{q}}
\end{gather*}
and
\begin{gather*}
 \widetilde D(x,n)=\left( \frac{q(p-1)}{p}\right)^{n+x} (1-q)^{-2j} \rFs{2}{1}{-x,-n}{-2j}{\frac{q-1}{q}}
\end{gather*}
 and they are biorthogonal with respect to the measure $w_{p,j}$. In particular, for the choice
\begin{gather*}
\frac 1 q=1-\frac 1 p
\end{gather*}
we have
\begin{gather*}
\begin{split}
D(x,n)=K(x,n;p)\qquad \text{and} \qquad \widetilde D(x,n)=(1-p)^{2j} D(x,n)
\end{split}
\end{gather*}
so that the biorthogonality relation recovers the orthogonality relation of Krawtchouk polynomials.

\item[$3.$] For ${\rm IRW}$ we have that
\begin{gather*}
D(x,n)=\rFs{2}{0}{-x,-n}{-}{\frac{1}{q}} \qquad \text{and} \qquad \widetilde D(x,n)= \left( -\frac{q}{p} \right)^{x+n} {\rm e}^{q} \rFs{2}{0}{-x,-n}{-}{\frac{1}{q}}
\end{gather*}
 and they are biorthogonal with respect to the measure $w_p$. In particular, for the choice
\begin{gather*}
q=-p
\end{gather*}
we have
\begin{gather*}
\begin{split}
D(x,n)=C(x,n;p)\qquad \text{and} \qquad \widetilde D(x,n)={\rm e}^{-p} D(x,n)
\end{split}
\end{gather*}
so that the biorthogonality relation recovers the orthogonality relation of Charlier polynomials.
\end{enumerate}
\end{Proposition}
\begin{proof}
As always we show the proof for the first item only.
Now we apply Proposition \ref{prop:orthogonality} with
\begin{gather*}
d_1(x,n) = D^{\rm cl}_{-q}(x,n)\qquad \text{and} \qquad d_2(x,n) = D^{\rm cl}_{-p}(x,n)
\end{gather*}
then, from Lemma \eqref{lem:biorthogonality}, the conditions \eqref{d12} are satisfied for $\mu_1=\mu_q$ and $\mu_2=\mu_p$.
We have
\begin{gather*}
D(x,n)=\big\langle D^{\rm cl}_{-q}(x,\cdot), D^{\rm cl}_{-q}(n,\cdot) \big\rangle_{q} = \rFs{2}{1}{-x,-n}{2k}{\frac{1}{q}} = M\left( x,n; \frac{q}{q-1} \right)
\end{gather*}
and
\begin{align*}
\widetilde D(x,n) &= \big\langle D^{\rm cl}_{-p}(\cdot,x), D^{\rm cl}_{-p}(\cdot,n) \big\rangle_{q}
 =(-p)^{-x-n} \sum_{z=n}^\infty \frac{ (2k)_z q^z }{z!} \frac{z! }{(z-x)!(2k)_x}\frac{z!}{(z-n)!(2k)_n} \\
& = \frac{(-p)^{-n-x} q^n}{(2k)_x (2k)_n}\sum_{m=0}^\infty \frac{ (2k)_{m+n} (m+n)! }{m! (m+n-x)! } q^m \\
& = \frac{(-p)^{-n-x} q^n n!}{(2k)_x (n-x)!} \sum_{m=0}^\infty \frac{(2k+n)_m (n+1)_m }{m! (n-x+1)_m } q^m \\
& = \frac{(-p)^{-n-x} q^n n!}{(2k)_x (n-x)!} \rFs{2}{1}{2k+n, n+1}{n-x+1}{q} .
\end{align*}
By applying the following ${}_2F_1$-transformations \cite[equations~(2.2.6), (2.3.14), (2.2.6)]{KLS}, we obtain
\begin{align*}
\rFs{2}{1}{2k+n, n+1}{n-x+1}{q} & = (1-q)^{-n-2k} \rFs{2}{1}{2k+n, -x}{n-x+1}{\frac{q}{q-1}} \\
& = (1-q)^{-n-2k} \frac{(2k)_x }{(-n)_x} \rFs{2}{1}{-x, 2k+n}{2k}{\frac{1}{1-q}} \\
& = (1-q)^{-n-x-2k} (-q)^x \frac{(2k)_x }{(-n)_x} \rFs{2}{1}{-x,-n}{2k}{\frac{1}{q}}.
\end{align*}
This gives
\begin{gather*}
\widetilde D(x,n) = \left(-\frac q{p(1-q)}\right)^{n+x} (1-q)^{-2k} \rFs{2}{1}{-x,-n}{2k}{\frac{1}{q}}
\end{gather*}
then, for $\frac 1 p=1-\frac 1 q$ we get
\begin{align*}
\widetilde D(x,n) &= (1-q)^{-2k} \rFs{2}{1}{-x,-n}{2k}{\frac{1}{q}}\\
&= (1-p)^{2k} \rFs{2}{1}{-x,-n}{2k}{1-\frac{1}{p}} = (1-p)^{2k} M(x,n; p).\tag*{\qed}
\end{align*}\renewcommand{\qed}{}
\end{proof}

\begin{Remark}From the previous proposition we have that we can write the Meixner-duality function $D(x,n)=M(x,n; p)$ in two forms in terms of scalar product:
\begin{gather}\label{scapro1}
D(x,n)=\big\langle D^{\rm cl}_{\sfrac{p}{1-p}}(x,\cdot), D^{\rm cl}_{\sfrac{p}{1-p}}(n,\cdot) \big\rangle_{{\frac{p}{p-1}}} = (1-p)^{-2k} \big\langle D^{\rm cl}_{-p}(\cdot,x), D^{\rm cl}_{-p}(\cdot,n) \big\rangle_{{\frac{p}{p-1}}} .
\end{gather}
\end{Remark}
We now write $D^{\rm cl}_{\sfrac{p}{1-p}}(x,y)$ and $D^{\rm cl}_{-p}(y,x)$ as a symmetry acting on the cheap duality. This allows us to write both expressions for $ D(x,n) $ in~\eqref{scapro1} via two symmetries ($S_{1}$ and $S_{2}$) acting on the cheap duality.
Before doing that, we will use the following lemma to justify some equality in the computation below.

\begin{Lemma}[duality of $K^{+}$ and $K^{-}$ via the cheap duality function] The operators~$K^{+}$ and~$K^{-}$ are dual via
\begin{gather*}
D^{\rm ch}_{1}(x,y)=\frac{x!\Gamma(2k)} {\Gamma(2k+x)} \delta_{x,y}.
\end{gather*}
\end{Lemma}
\begin{proof}This follows from
\begin{gather*}
\big( K^{+}D^{\rm ch}_{1} (\cdot, y) \big) (x) = \frac{y!\Gamma(2k)}{\Gamma(2k+y-1)} = \big( K^{-}D^{\rm ch}_{1}(x, \cdot)\big) (y) .\tag*{\qed}
\end{gather*}\renewcommand{\qed}{}
\end{proof}

As a consequence of the above relation, we have the following corollary.
\begin{Corollary}\label{cor-lemma}
The operators ${\rm e}^{K^{+}}$ and $ {\rm e}^{K^{-}} $ are also in duality via ${D_{1}^{\rm ch}}$. Moreover, we can choose parameter $\alpha$, $\beta$ on the exponentials and $\lambda$ on ${D_{1}^{\rm ch}}$ such that the relation
\begin{gather*}
\Big({\rm e}^{\alpha K^{-}} D^{\rm ch}_{\frac{1}{\lambda}}(\cdot ,y) \Big) (x)= \Big({\rm e}^{\beta K^{+}} D_{\frac{1}{\lambda}}^{\rm ch}(x,\cdot) \Big) (y)
\end{gather*}
is always true for any $\alpha$, $\beta$ and $\lambda \in \mathbb{R}$ that satisfy $\alpha= \lambda \beta$.
\end{Corollary}

To make notation simpler we write $ K^{-}_{1}$ (resp.~$ K^{-}_{2} $) for the action of $K^{-} $ on the first (resp.\ second) variable and same for $K^{+} $. Let's now investigate the two symmetries associated to the self-duality function in equation~\eqref{scapro1}.
\begin{Proposition}[two ways of expressing orthogonal polynomials] Let $D$ be the self-duality function given by the two scalar products in equation~\eqref{scapro1}, then $D$ can be written as a symmetry acting on $D^{\rm ch}$. In particular, we have that
\begin{gather}\label{s1}
 D(x,n)
= {\rm e}^{K^{-}_{1}} {\rm e}^{\frac{p}{p-1}K^{+}_{1}} D^{\rm ch}_{\sfrac{p}{p-1}}(x,n)
\end{gather}for the first scalar product in~\eqref{scapro1}, and
\begin{gather}\label{s2}
 D(x,n)=(1-p)^{-2k} {\rm e}^{-pK^{+}_{1}} {\rm e}^{\frac{1}{1-p}K^{-}_{1}} D^{\rm ch}_{p(p-1)}(x, n)
\end{gather}for the second one.
\end{Proposition}
\begin{proof}For the first scalar product in \eqref{scapro1} we have that, using the first item of Proposition \ref{clalemm}
\begin{gather}\label{sp}
D(x,n)= \big\langle D^{\rm cl}_{\sfrac{p}{1-p}}(x,\cdot), D^{\rm cl}_{\sfrac{p}{1-p}}(n,\cdot) \big\rangle_{{\frac{p}{p-1}}}
= \big\langle {\rm e}^{K^{-}_{1}} D^{\rm ch}_{\sfrac{p}{1-p}}(x, \cdot), {\rm e}^{K^{-}_{1}} D^{\rm ch}_{\sfrac{p}{1-p}}(n, \cdot) \big\rangle_{{\frac{p}{p-1}}} .
\end{gather}The action of ${\rm e}^{K^{-}_{1}} $ only affect the $x$ variable and it can be placed outside the scalar product, moreover by Corollary~\ref{cor-lemma}, the first scalar product in \eqref{sp} is equal to
\begin{gather*}
{\rm e}^{K^{-}_{1}} \big\langle D^{\rm ch}_{\sfrac{p}{1-p}}(x, \cdot), {\rm e}^{\frac{p}{1-p}K^{+}_{2}} D^{\rm ch}_{\sfrac{p}{1-p}}(n, \cdot) \big\rangle_{{\frac{p}{p-1}}} . \end{gather*}
The adjoint of $K^{+}$ with respect to $\mu_{\sfrac{p}{p-1}}$ is $\frac{p-1}{p}K^{-}$ and the above quantity becomes
\begin{gather*}
{\rm e}^{K^{-}_{1}} \big\langle {\rm e}^{-K^{-}_{2}} D^{\rm ch}_{\sfrac{p}{1-p}}(x, \cdot), D^{\rm ch}_{\sfrac{p}{1-p}}(n, \cdot) \big\rangle_{{\frac{p}{p-1}}}.
\end{gather*}
Using again Corollary~\ref{cor-lemma} we finally get
\begin{gather*}
{\rm e}^{K^{-}_{1}} {\rm e}^{\frac{p}{p-1}K^{+}_{1}} \big\langle D^{\rm ch}_{\sfrac{p}{1-p}}(x, \cdot), D^{\rm ch}_{\sfrac{p}{1-p}}(n, \cdot) \big\rangle_{{\frac{p}{p-1}}}
= {\rm e}^{K^{-}_{1}} {\rm e}^{\frac{p}{p-1}K^{+}_{1}} D^{\rm ch}_{\sfrac{p}{p-1}}(x,n) ,
\end{gather*}
where we computed the scalar product to get the symmetry in~\eqref{s1}.

In a similar fashion, for the second scalar product in \eqref{sp} we have that
\begin{align*}
D(x,n) & = (1-p)^{-2k} \big\langle D^{\rm cl}_{-p}(\cdot,x), D^{\rm cl}_{-p}(\cdot, n) \big\rangle_{{\frac{p}{p-1}}}\\
& = (1-p)^{-2k} \big\langle {\rm e}^{K^{-}_{2}} D^{\rm ch}_{-p}(x, \cdot), {\rm e}^{K^{-}_{2}}D^{\rm ch}_{-p}(n, \cdot) \big\rangle_{{\frac{p}{p-1}}}
\end{align*}
by Corollary~\ref{cor-lemma}. Then, by considering the adjoint of $K^{-}$ with respect to $\mu_{\sfrac{p}{p-1}}$, the above expression reads
\begin{gather*}
 (1-p)^{-2k} {\rm e}^{-pK^{+}_{1}} \big\langle {\rm e}^{\frac{p}{p-1}K^{+}_{2}} D^{\rm ch}_{-p}(x, \cdot), D^{\rm ch}_{-p}(n, \cdot) \big\rangle_{{\frac{p}{p-1}}}
\end{gather*}
and again, by Corollary~\ref{cor-lemma} we have the symmetry in~\eqref{s2}
\begin{gather*}
 (1-p)^{-2k} {\rm e}^{-pK^{+}_{1}} {\rm e}^{\frac{1}{1-p}K^{-}_{1}} \big\langle D^{\rm ch}_{-p}(x, \cdot), D^{\rm ch}_{-p}(n, \cdot) \big\rangle_{{\frac{p}{p-1}}}\\
 \qquad{} = (1-p)^{-2k} {\rm e}^{-pK^{+}_{1}} {\rm e}^{\frac{1}{1-p}K^{-}_{1}} D^{\rm ch}_{p(p-1)}(x, n) .\tag*{\qed}
\end{gather*}\renewcommand{\qed}{}
\end{proof}

\begin{Remark}[a commutation relation for the exponentials] The expression for $ D(x,n) $ in~\eqref{s2} can be written as an action on $ D^{\rm ch}_{p(p-1)}(x,n) $ as follows,
\begin{gather*}
D(x,n) = (1-p)^{-2k} {\rm e}^{-pK^{+}_{1}} {\rm e}^{\frac{1}{1-p}K^{-}_{1}} D^{\rm ch}_{p(p-1)}(x, n) = {\rm e}^{-pK^{+}_{1}} {\rm e}^{\frac{1}{1-p}K^{-}_{1}} (1-p)^{-2K^{0}_{1}} D^{\rm ch}_{\sfrac{p}{p-1}}(x,n) ,
\end{gather*}
comparing with \eqref{s1} allows us to infer the following relation
\begin{gather}\label{rose}
{\rm e}^{K^{-}} {\rm e}^{\frac{p}{p-1}K^{+}} = {\rm e}^{-pK^{+}} {\rm e}^{\frac{1}{1-p}K^{-}} (1-p)^{-2K^{0}} .
\end{gather}Relation in \eqref{rose} is found in \cite[Remark~3.2]{Rosengren} adapted to the $\mathfrak{su}(1,1)$ Lie algebra generators.
\end{Remark}
We now do some consideration about the relations among the symmetries in equations \eqref{s1} and \eqref{s2} with the one obtained in Proposition~\ref{factsymm}. In order to compare their action, we first realize that
\begin{gather*}
 D^{\rm ch}_{\sfrac{p}{p-1}}(x,n)= (p-1)^{x} D^{\rm ch}_{p} (x,y) ,
\end{gather*}
so that their action on $D^{\rm ch}_{p}$ reads
\begin{gather*}
(p-1)^k D(x,n)=S_{1}D^{\rm ch}_{p} (x,y) := {\rm e}^{K^{-}} {\rm e}^{\frac{p}{p-1}K^{+}} (p-1)^{K^{0}} D^{\rm ch}_{p} (x,y)
\end{gather*}
and
\begin{gather*}
(p-1)^k D(x,n)=S_{2}D^{\rm ch}_{p}(x,y) := {\rm e}^{-pK^{+}} {\rm e}^{\frac{1}{1-p}K^{-}} (p-1)^{-K^{0}} D^{\rm ch}_{p}(x,y) .
\end{gather*}
It is easy to verify that both $S_{1}$ and $S_{2}$ are unitary operators on $\mathsf L^2(w_{p,k})$ since the self-duality functions $ (p-1)^k D(x,n) $ and $ D^{\rm ch}_{p}(x,y) $ have the same norm. One could also check it via the generating function approach.

\begin{Remark}[equivalence of the symmetries] As a final remark we mention that the symmetry from Proposition \ref{factsymm} and $S_1$ are, as expected, the same.
If we consider their action on $ D^{\rm ch}_{p} (x,y)$, then
\begin{gather*}
\big( {\rm e}^{K^{-}} {\rm e}^{\frac{p}{p-1}K^{+}} (p-1)^{K^{0}} D^{\rm ch}_{p} (\cdot ,y)\big) (x) = (p-1)^k M(x,y;p),
\end{gather*}
while
\begin{gather*}
\big( {\rm e}^{K^{-}}(p-1)^{K^{0}}{\rm e}^{p K^{+}} D^{\rm ch}_{p} (\cdot,y) \big) (x)= (p-1)^k M(x,y;p)
\end{gather*}
inferring the commutation relation between $K^{+}$ and $ K^{0} $:
\begin{gather*}
{\rm e}^{\frac{p}{p-1}K^{+}} (p-1)^{K^{0}} = (p-1)^{K^{0}}{\rm e}^{p K^{+}} .
\end{gather*}
\end{Remark}

\appendix

\section{Appendix. Proof of Proposition \ref{prop11}} \label{APP}
We have to show that
\begin{gather*}
{\rm e}^{K^{-}}{\rm e}^{\log(p-1) K^{0}}{\rm e}^{pK^{+}} \big( D^{\rm ch}_{p} ( x, \cdot) \big) (y) = D^{\rm or}_{p}(x,y),
\end{gather*}
where
\begin{gather*}
D^{\rm ch}_{p} (x,y) =\frac{y! \Gamma(2k)}{\Gamma(2k+y)} p^{-y} \qquad \text{and}\qquad D^{\rm or}_{p}(x,y) = (p-1)^{k} \rFs{2}{1}{-x, -y}{2k }{1-\frac{1}{p}}.
\end{gather*}

We start by acting with the inner operator on the $y$ variable of $D^{\rm ch}_{p}$:
\begin{gather*}
 {\rm e}^{K^{-}}{\rm e}^{\log(p-1) K^{0}}{\rm e}^{pK^{+}} \big( D^{\rm ch}_{p} ( x, \cdot) \big) (y) \\
 \qquad{} = {\rm e}^{K^{-}}{\rm e}^{\log(p-1) K^{0}}\sum_{i=0}^{\infty} \frac{p^{i}}{i!} (K^{+}) ^{i} \frac{x! \Gamma(2k) }{\Gamma(2k+x)} (p) ^{-x} \delta_{x,y} \\
 \qquad{} = {\rm e}^{K^{-}}{\rm e}^{\log(p-1) K^{0}}\sum_{i=0}^{\infty} \frac{p^{i}}{i!} \frac{\Gamma(2k+y+i)}{\Gamma(2k+y)} \frac{x! \Gamma(2k) }{\Gamma(2k+x)} p^{.x} \delta_{x,y+i}
\end{gather*}
since the action of the $i^{\rm th}$ power of $ K^{+} $ is $(K^{+}) ^{i} f(y) = \frac{\Gamma(2k+y+i)}{\Gamma(2k+y)} f(y+i) $. The action of $ K^{0} $ is diagonal and so we have
\begin{gather*}
{\rm e}^{K^{-}}{\rm e}^{\log(p-1) K^{0}} p^{x-y} \frac{x!}{(x-y)!}\frac{\Gamma(2k)}{\Gamma(2k+y)} p^{-x} \mathbf{1}_{\lbrace x \geq y \rbrace } \\
\qquad{}={\rm e}^{K^{-}}(p-1) ^{y+k} \frac{x!}{(x-y)!}\frac{\Gamma(2k)}{\Gamma(2k+y)} \left( \frac{1}{p}\right)^{y} \mathbf{1}_{\lbrace x \geq y \rbrace} .
\end{gather*}
Finally we use the action of the $i^{\rm th}$ power of $ K^{-} $, i.e., $( K^{-}) ^{i} f(y) = \frac{y!}{(y-i)!} f(y-i) $,
\begin{gather*}
(p-1)^{k} \sum_{i=0}^{\infty} \frac{(K^{-}) ^{i}}{i!} \left( \frac{p-1}{p}\right) ^{y} \frac{x!}{(x-y)!}\frac{\Gamma(2k)}{\Gamma(2k+y)} \mathbf{1}_{\lbrace x \geq y \rbrace} \\
\qquad{} = (p-1)^{k} \sum_{i=0}^{\infty} \frac{1}{i!} \frac{y!}{(y-i)!}\left(\frac{p-1}{p} \right) ^{y-i} \frac{x!}{(x-y+i)!} \frac{\Gamma(2k)}{\Gamma(2k+y-i)} \mathbf{1}_{\lbrace x \geq y-i \rbrace} \mathbf{1}_{\lbrace i \leq y \rbrace} \\
\qquad{}= (p-1)^{k} \sum_{i=0 \vee (y-x)}^{y} \frac{1}{i!} \frac{y!}{(y-i)!}\left( \frac{p-1}{p} \right) ^{y-i} \frac{x!}{(x-y+i)!} \frac{\Gamma(2k)}{\Gamma(2k+y-i)} \\
\qquad{} = (p-1)^{k} \sum_{s=0}^{x \wedge y} \frac{x!}{(x-s)!} \frac{y!}{(y-s)!} \frac{1}{s!}\left( \frac{p-1}{p}\right) ^{s} \frac{\Gamma(2k)}{\Gamma(2k+s)} \\
\qquad{} = (p-1)^{k} M(x,y;p) = D^{\rm or}_{p} (x,y).
\end{gather*}
Where we performed the change of variable $y-i=s$ at the end. Note that, up to the constant $(p-1)^{k}$, the last sum is the closed form of Meixner polynomials of parameter~$p$ and $2k$, variable~$x$ and degree~$y$. Since the result is symmetric in~$x$ and~$y$ then the action on the $x$ variable would produce the same result.

\subsection*{Acknowledgements}

We thank the anonymous referees for their input which helped to improve the paper.

\pdfbookmark[1]{References}{ref}
\LastPageEnding

\end{document}